\documentclass[11pt]{amsart}
\usepackage{a4wide}
\usepackage[T1]{fontenc}
\usepackage{amssymb,amsmath,amsthm,latexsym}
\usepackage{mathrsfs}
\usepackage[usenames,dvipsnames]{color}
\usepackage{euscript}
\usepackage{graphicx}
\usepackage{mdwlist}
\usepackage{enumerate}
\usepackage{mathtools,dsfont,wasysym}
\usepackage{bbm}
\usepackage{stmaryrd}
\usepackage{centernot}
\usepackage{todonotes}
\usepackage{hyperref}

\makeatletter
\@namedef{subjclassname@2020}{\textup{2020} Mathematics Subject Classification}
\makeatother

\newtheorem{theorem}{Theorem}[section]
\newtheorem{thm}{Theorem}[section]

\newtheorem{lem}[theorem]{Lemma}

\newtheorem{cor}[theorem]{Corollary}

\newcounter{maintheorem}

\theoremstyle{remark}

\newtheorem{rmk}[theorem]{Remark}
\theoremstyle{definition}

\newtheorem{defn}[theorem]{Definition}

\numberwithin{equation}{section}
\makeatother

\newcommand{\nn}[1]{{\left\vert\kern-0.25ex\left\vert\kern-0.25ex\left\vert #1
\right\vert\kern-0.25ex\right\vert\kern-0.25ex\right\vert}}

\newcommand{\revcoloneqq}{\mathrel{=\vcentcolon}}
\renewcommand{\leq}{\leqslant}
\renewcommand{\geq}{\geqslant}

\newcounter{smallromans}

{\end{list}}

\begin{document}
\title[]{Nonexistence of Solutions to classes of parabolic inequalities in the Riemannian setting}

\author[D.-E. von Criegern]{Dorothea-Enrica von Criegern}
\address[D.-E. von Criegern]{Politecnico di Milano, Dipartimento di Matematica, Piazza Leonardo da Vinci 32, 20133 Milano, Italy}
\email{dorotheaenrica.von@polimi.it}

\author[G.~Grillo]{Gabriele Grillo}
\address[G.~Grillo]{Politecnico di Milano, Dipartimento di Matematica, Piazza Leonardo da Vinci 32, 20133 Milano, Italy}
\email{gabriele.grillo@polimi.it}

\author[D.D.~Monticelli]{Dario D. Monticelli}
\address[D.D.~Monticelli]{Politecnico di Milano, Dipartimento di Matematica, Piazza Leonardo da Vinci 32, 20133 Milano, Italy}
\email{dario.monticelli@polimi.it}

%\thanks{The research of the authors has been partially supported by GNAMPA (INdAM -- Istituto Nazionale di Alta Matematica).}

%\keywords{Porous medium equation, smoothing estimates, Green function, Riemannian manifolds, Ricci curvature, potential estimates}
%\date{\today}

\begin{abstract} We establish conditions for nonexistence of global solutions for a class of quasilinear parabolic problems with a potential on complete, non-compact Riemannian manifolds, including the Porous Medium Equation and the p-Laplacian with a potential term. Our results reveal the interplay between the manifold's geometry, the power nonlinearity, and the potential's behavior at infinity. Using a test function argument, we identify explicit parameter ranges where nonexistence holds.

\end{abstract}
\maketitle

\section{Introduction}

We investigate the nonexistence of global, nonnegative, nontrivial weak solutions (in the sense of Definition~\ref{thm:Defn} below) to parabolic differential inequalities of the type

\begin{equation}\label{main eq}
    \partial_t u \geq \frac{1}{a(x)} \mathrm{div}\Big( a(x) u^{m-1} |\nabla u|^{p-2} f(|\nabla u|) \nabla u \Big) + V(x,t) u^q\;\;\;\;\;\mathrm{in}\;M\times(0,\infty),
\end{equation}

\noindent where $M$ is a complete, non-compact, $N$-dimensional Riemannian manifold with metric given by $g$. The operators $\mathrm{div}$ and $\nabla$ denote the divergence, respectively the gradient with respect to $g$. We further assume $p,q > 1$, $m\geq1$, and that the potential $V\in L^1_{\mathrm{loc}}(M\times(0,\infty))$ satisfies $V>0$ a.e. in $M \times(0,\infty)$. In addition, $0\leq f \leq K$ for some $K>0$, $a\in \mathrm{Lip}_{\mathrm{loc}}(M)$ with $a>0$ a.e. in $M \times(0,\infty)$.

Clearly, as special cases we deal with reaction-diffusion equations in which the diffusion part is driven by the $p$-Laplacian, or by the porous medium diffusion, or by the doubly nonlinear evolution equation, see Section \ref{appl} for an explicit discussion and examples. We stress that our conditions, in particular the fact that $f$ is required to satisfy only $0\leq f \leq K$, are general enough to deal with other classes of evolution equations, like e.g. evolution of graphs by mean curvature, see Example \bf I4 \rm below.

In the Euclidean setting, reaction-diffusion equations have a long history, starting from the work of Fujita \cite{Fujita, Fuj2}, for the differential equation
\[
u_t=\Delta u+u^q.
\]
in $\mathbb R^N\times(0,+\infty)$. It is well known from his work and, for example, from \cite{Haya, KST}, that if $q\in\left(1, +\tfrac2N\right]$, no nontrivial nonnegative global solution exists, since all such solutions blow up in finite time; whereas if $q>\tfrac2N$, solutions corresponding to sufficiently small nonnegative initial data exist globally in time. It would be impossible to summarize the huge amount of research stemming from such seminal results, so we limit ourselves to quoting, without claim to generality, some further classical work like \cite{Galak1, Galak2, GalakLev, Lev, MitPohoz359, MitPohozAbsence, MitPohozMilan, PoTe0, PuTe, Weiss, Weiss2, SGKM} in which one can find discussions of various different but related differential inequalities driven by the $p$-Laplacian, by the porous medium diffusion or by doubly nonlinear operators. %We comment that, specifically in the latter case, a number of questions are open in the fast diffusion case, i.e. when dealing with the equation
%\[
%u_t=\Delta u^m+u^q
%\]
%under the assumption $m\in(0,1)$, in which the diffusion is so fast to make not easily prediction the outcome of the competition with the source term $u^q$, especially when, for $N\ge3$, one has $m\in\left(0,\tfrac{N-2}N\right)$, see e.g. \cite{GV} for some of the existing results.

The analysis of similar problems in the setting of Riemannian manifolds is more recent. A major contribution was given in \cite{Zhang}, in which the heat equation is treated together with the porous medium equation and, if the parameter $m$ in the equation $u_t=\Delta u^m+u^\sigma$ is sufficiently close to one, the fast diffusion one. A key feature of the results proved there is the request of polynomial volume growth, which is required with matching upper and lower bounds of the form
\[
\textrm{Vol}\,(B_r(x))\le C\,r^\alpha\ \ \ \forall x\in M, \forall r\ge1,
\]
where $B_r(x)$ is the Riemannian ball centered at $x\in M$ and of radius $r$, and Vol denotes the Riemannian volume. Besides, when dealing with the reaction-diffusion equation related to heat diffusion, namely with
\begin{equation}\label{heat}
u_t=\Delta u+V(x)u^q,
\end{equation}
a strict assumption on $V(x)$ is required, in the sense that $V(x)$ needs to satisfy matching upper a lower bound of the form
\[
cd(x,o)^m\le V(x)\le Vd(x,o)^m
\]
for a given $m>-2$, $o\in M$, whenever $d(x,o)$ is sufficiently large, where $d$ denotes the Riemannian distance, and a further additional condition on the volume density. Later on, some of these condition have been relaxed in \cite{GSXX}, again for solutions to \eqref{heat}. In \cite{dariospaper}, nonexistence results for nonnegative \it supersolutions \rm to the reaction diffusion problem for the $p$-Laplacian
\begin{equation*}%\label{plapl}
u_t=\Delta_pu+u^q,
\end{equation*}
where $\Delta_p=\nabla\cdot(|\nabla u|^{p-2}\nabla u)$, has been derived, under suitably weighted  growth assumptions on the volumes of Riemanian balls, as the radius tends to infinity. Finally, the recent work \cite{SX} deals with supersolutions to the doubly nonlinear reaction-diffusion equation
\begin{equation}\label{doubly}
u_t=\Delta_p(u^m)+u^q.
\end{equation}
 %, under suitable restrictions on $p, m, q$, namely $p>0$, $m\not=0$, $p+m-2>0$, $q>\textrm{min}\,(1, p+m-2)$. sembra più generale del nostro, lo diciamo in dettaglio o no? Poi ovviamente le nostre condizioni sono meglio.
Volumes are required to satisfy a bound of the type
\[
\textrm{Vol}\,(B_r(x))\le C\,r^\alpha(\log r)^\beta\ \ \ \forall x\in M, \forall r\ge r_0>1,
\]
%non sono sicuro che non servano solo le palle centrate in un punto fissato
where $\alpha, \beta$ are related to the parameters $p,m,q$ in \eqref{doubly}. It is important to comment that the results of \cite{SX} hold for \it strictly positive \rm solutions, since the choice of test function made there makes sense only in that case. It is well-known that in the porous medium case, even with appropriate forcing, there exist solutions which are compactly supported for all time, hence the extension to \it nonnegative \rm solutions is relevant and in our view natural.

We provide here results for the more general differential inequality \eqref{main eq}, under assumptions of a different, and in several crucial aspects weaker, type, with an approach which is similar in spirit to the one used in \cite{GSXX, dariospaper}. As in the previously mentioned papers we do not make any use of curvature assumptions, nor further strict geometric conditions e.g. on the possible cut-locus, but only require conditions related to volume growth, with an emphasis on assumptions regarding weighted volume growth, where the weight is related to the potential $V$ in \eqref{main eq}, see the conditions in Section \ref{thm:sec11}. We find particularly relevant that, due to the presence of a potential $V$ in \eqref{main eq}, \it volumes need not be required to satisfy a polynomial upper bound\rm, see the bounds in \bf HP1-2 \rm below, provided $V$ compensates for a possible faster growth. In fact, volume growth can be arbitrarily fast in manifolds having negative sectional curvatures, and very different results can hold in such setting, see the recent works \cite{BPT, GMP0, GMP, P1, P2}. Such setting can be considered here provided $V$ is sufficiently large, in an integral sense, see again assumptions \bf HP1-2\rm. We also stress the presence of the coefficients $a(x), f(\nabla u)$ in \eqref{main eq}, and the fact that the potential $V$ can be time-dependent as well. Finally, as already mentioned, we stress that our assumptions are general enough to deal with different classes of evolution equations,  e.g. with evolution of graphs by mean curvature, see Example \bf I4 \rm below.

We comment that we will not address here the complementary issue of global existence of solutions to parabolic quasilinear problems similar to \eqref{main eq} on $\mathbb R^n$ or on Riemannian manifolds, which is similarly widely studied in the literature, see e.g. without claim for completeness, the papers \cite{GMP2, GMP3, MT, MTS1, MTS2, MOP}.

\subsection{Main Results}\label{thm:sec11}

We shall use the notation, here and in the sequel without further comment, $S \coloneq M\times(0,\infty)$. Fix some $x_0\in M$ and denote by $r(x)\coloneq \mathrm{dist}(x_0,x)$ the Riemannian distance between $x_0$ and $x$. For each $R > 0$, $\theta_1, \theta_2 \geq 1$, set
\begin{equation}\label{annulus}
    E_R \coloneq \{(x,t)\in S \;:\; r(x)^{\theta_2}+t^{\theta_1} \leq R^{\theta_2}\}.
\end{equation}
We further introduce the following constants:
\begin{equation}\label{s}
\begin{split}
    \bar{s}_1 &\ \coloneq \frac{q}{q-1}\theta_2, \;\;\;\;\; \bar{s}_2 \coloneq \frac{1}{q-1}\\
      \bar{s}_3 &\ \coloneq  \frac{pq}{q-p-m+2}\theta_2, \;\;\;\;\; \bar{s}_4 \coloneq \frac{p+m-2}{q-p-m+2}. \\
\end{split}
\end{equation}

The following two sets of conditions are the main hypotheses under which we will prove the nonexistence results of global, nonnegative, nontrivial weak solutions to $\eqref{main eq}$; by $\mu_a$ we denote the weighted measure $a\mu$, where $\mu$ is the Riemannian measure on $M$. Throughout this paper, $C$ will denote a generic constant that may vary from line to line, but is independent of all relevant parameters.

\begin{itemize}

    \item[(\textbf{HP}$\mathbf{1}$)] Assume that there exist constants $\theta_1\geq 1, \theta_2\geq1, C_0 > 0, C > 0$, and $R_0 > 1$, $\epsilon_0 > 0$ such that for all $R>R_0$ and all $0<\epsilon < \epsilon_0$, one has

    \begin{itemize}
        \item[(i)]
        \begin{equation}\label{HP11}
            \int\int_{E_{2R} \setminus E_R} t^{(\theta_1-1)\left(\frac{q}{q-1}-\epsilon\right)}\; V^{-\frac{1}{q-1}+\epsilon}\;d\mu_a dt
            \leq C R^{\bar{s}_1+C_0\epsilon} \log(R)^{s_2};
        \end{equation}

        \item[(ii)]
        \begin{equation}\label{HP12}
            \int\int_{E_{2R} \setminus E_R} r(x)^{(\theta_2-1)p\left(\frac{q}{q-p-m+2}-\epsilon\right)}\; V^{-\frac{p+m-2}{q-p-m+2}+\epsilon}\;d\mu_a dt
            \leq C R^{\bar{s}_3+C_0\epsilon} \log(R)^{s_4}
        \end{equation}

        for some $0\leq s_2 < \bar{s}_2$, $0\leq s_4 < \bar{s}_4$, $E_R$ being as in \eqref{annulus} and $\bar{s}_1,\ldots, \bar{s}_4$ as in \eqref{s}.
    \end{itemize}

    \medskip

    \item[(\textbf{HP}$\mathbf{2}$)] Assume that there exists constants $\theta_1\geq 1, \theta_2\geq1, C_0 > 0, C > 0$, and $R_0 > 1$, $\epsilon_0 > 0$ such that for all $R>R_0$ and all $0<\epsilon < \epsilon_0$, one has

    \begin{itemize}
        \item[(i)]
        \begin{equation}\label{HP21}
            \int\int_{E_{2R} \setminus E_R} t^{(\theta_1-1)\left(\frac{q}{q-1}-\epsilon\right)}\; V^{-\frac{1}{q-1}+\epsilon}\;d\mu_a dt
            \leq C R^{\bar{s}_1+C_0\epsilon} \log(R)^{\bar{s}_2},
        \end{equation}
        \item[(ii)]
        \begin{equation}\label{HP22}
           \int\int_{E_{2R} \setminus E_R} r(x)^{(\theta_2-1)p\left(\frac{q}{q-p-m+2}-\epsilon\right)}\; V^{-\frac{p+m-2}{q-p-m+2}+\epsilon}\;d\mu_a dt
            \leq C R^{\bar{s}_3+C_0\epsilon} \log(R)^{\bar{s}_4},
        \end{equation}
        \item[(iii)]
        \begin{equation}\label{HP23}
            \int\int_{E_{2R} \setminus E_R} r(x)^{(\theta_2-1)p\left(\frac{q}{q-p-m+2}+\epsilon\right)}\; V^{-\frac{p+m-2}{q-p-m+2}-\epsilon}\;d\mu_a dt
            \leq C R^{\bar{s}_3+C_0\epsilon} \log(R)^{\bar{s}_4},
        \end{equation}
        \end{itemize}
        $E_R$ being as in \eqref{annulus} and $\bar{s}_1,\ldots,\bar{s}_4$ as in \eqref{s}.
\end{itemize}

\begin{rmk}
For all $(x,t)\in E_{2R} \setminus E_R$, we have $t\leq CR^{\theta_2/\theta_1}$ and $r(x)\leq C R$. Hence, the estimates \eqref{HP21}-\eqref{HP23} in \textbf{HP}$\mathbf{2}$ hold in particular when the potential $V$ satisfies the following growth conditions for all $\epsilon > 0$ sufficiently small and all $R > 1$ sufficiently large:

        \begin{equation*}
            \int\int_{E_{2R} \setminus E_R}  V^{-\frac{1}{q-1}+\epsilon}\;d\mu_a dt
            \leq C R^{\frac{\bar{s}_1}{\theta_1}+C_1\epsilon} \log(R)^{\bar{s}_2},
        \end{equation*}
        \begin{equation*}
           \int\int_{E_{2R} \setminus E_R}  V^{-\frac{p+m-2}{q-p-m+2}+\epsilon}\;d\mu_a dt
            \leq C R^{\frac{\bar{s}_3}{\theta_2}+C_1\epsilon} \log(R)^{\bar{s}_4},
        \end{equation*}
        \begin{equation*}
            \int\int_{E_{2R} \setminus E_R}  V^{-\frac{p+m-2}{q-p-m+2}-\epsilon}\;d\mu_a dt
             \leq C R^{\frac{\bar{s}_3}{\theta_2}+C_1\epsilon} \log(R)^{\bar{s}_4},
        \end{equation*}
        for some $C_1\ge0$. Analogous growth conditions hold in the setting of \textbf{HP}$\mathbf{1}$ with appropriate modifications to the exponents.
\end{rmk}

\begin{rmk}\label{thm:rmkepsilon}
    By Fatou's Lemma, we can pass to the limit $\epsilon\rightarrow0$ in all the growth estimates in \textbf{HP}$\mathbf{1}$ and \textbf{HP}$\mathbf{2}$. The above conditions then also hold for $\epsilon=0$.
\end{rmk}

We introduce the following notion of weak solutions:

\begin{defn}\label{thm:Defn}
    Let $p,q>1, m\geq1$ and let $V>0$ a.e. in $M\times (0,\infty)$, $V \in L^1_{\mathrm{loc}}(M\times (0,\infty))$. Let further $0\leq f\leq K$ for some $K>0$ and $a\in \mathrm{Lip}_{\mathrm{loc}}(M)$, $a>0$ a.e. We say that $u$ is a weak solution to Inequality $\eqref{main eq}$ if $u \geq 0$ a.e. in $M \times (0, \infty)$, if $u \in L^q_{\mathrm{loc}}(M \times (0, \infty); V d\mu_a dt)$, and if $u^{p+m-1}$, $|\nabla u|^p u^{m-1}$, and $u(\partial_t u)$ belong to $L^1_{\mathrm{loc}}(M \times (0, \infty); d\mu_a dt)$. Furthermore, for every nonnegative test function $\psi \in \mathrm{Lip}_{\mathrm{loc}}(M \times (0, \infty))$ with compact support, the following inequality must hold
    \begin{equation}\label{equivdefn}
    \begin{split}
        \int_0^{\infty}\int_M \psi u^q V \;d\mu_a dt \leq   &\ \int_0^{\infty}\int_M   \langle \nabla \psi, \nabla u \rangle  u^{m-1} |\nabla u|^{p-2} f(|\nabla u|)\;d\mu_a dt\\
       &\ - \int_0^{\infty}\int_M (\partial_t\psi)u \;d\mu_a dt,
        \end{split}
    \end{equation}
    where $\mu_a$ denotes the weighted measure $a\mu$ in $M$.
\end{defn}

% \begin{rmk}
%     Note that, for $a\in\mathrm{Lip}_{loc}(M)$, $u
%     $ is a weak solution to Problem $\eqref{main eq}$ if and only if for every $\psi\in W^{1,l}(M\times (0,\infty))\cap L^{\infty}(M\times (0,\infty)), \psi \geq0$ with compact support in $M\times (0,\infty)$, one has
%     \begin{equation*}
%     \begin{split}
%         \int_0^{\infty}\int_M \psi u^q V \;d\mu dt \leq &\  \int_0^{\infty}\int_M \Big\langle \nabla \Big(\frac{\psi}{a}\Big), \nabla u \Big\rangle a u^{m-1} |\nabla u|^{p-2} f(|\nabla u|)\;d\mu dt\\
%        &\ - \int_0^{\infty}\int_M (\partial_t\psi)u \;d\mu dt.
%     \end{split}
%     \end{equation*}
%    If $ \psi $ is a nonnegative function that belongs to $ \mathrm{Lip}_{\mathrm{loc}}(M\times (0,\infty)) $ with compact support, then multiplying $ \psi $ by $ a $ results in a valid test function, as described in Definition~\ref{thm:Defn}. Conversely, if $ \psi $ is already a valid test function, dividing it by $ a $ ensures that it remains in $ \mathrm{Lip}_{\mathrm{loc}}(M\times (0,\infty)) $ with compact support.
% \end{rmk}

We will prove the following nonexistence results:

\begin{thm}\label{thm:Thm1}
Let $p>1$, $m\geq1$, $q>\max(p+m-2, 1)$ and $V>0$ a.e. in $M\times (0,\infty)$, $V\in L^1_{loc}(M\times (0,\infty))$. Let further $0\leq f \leq K$ for some $K>0$ and $a\in\mathrm{Lip}_{\mathrm{loc}}(M)$, $a>0$ a.e. in $M$. If $u$ is a global nonnegative solution to Problem $\eqref{main eq}$ and Condition \textbf{HP}$\mathbf{1}$ holds, then $u=0$ a.e. in $M\times (0,\infty)$.
\end{thm}

\begin{thm}\label{thm:Thm2}
Let $p>1$, $m\geq1$, $q>\max(p+m-2, 1)$ and $V>0$ a.e. in $M\times (0,\infty)$, $V\in L^1_{loc}(M\times (0,\infty))$. Let further $0\leq f \leq K$ for some $K>0$ and $a\in\mathrm{Lip}_{\mathrm{loc}}(M)$, $a>0$ a.e. in $M$. If $u$ is a global nonnegative solution to Problem $\eqref{main eq}$ and Condition \textbf{HP}$\mathbf{2}$ holds, then $u=0$ a.e. in $M\times (0,\infty)$.
\end{thm}

In order to prove the above theorems, we will use a test function argument based on Definition~\ref{thm:Defn}, aiming to bound the $u$-dependent terms by integrals independent of $u$, which will allow us to apply Hypotheses \textbf{HP}$\mathbf{1}$ and \textbf{HP}$\mathbf{2}$.

\subsection{Applications}\label{appl}

The nonexistence results in Theorems~\ref{thm:Thm1} and~\ref{thm:Thm2} imply nonexistence of global solutions for certain parabolic differential (in)equalities of the form in $\eqref{main eq}$. We highlight four well-studied equations, including reaction-diffusion equation in which the diffusion part correspond to the Porous Medium Equation, see \cite{V}, and the $p$-Laplacian evolution equation, see e.g. \cite{DGV}.

In the following, we let $M$ be a complete and non-compact Riemannian manifold.

\begin{itemize}
    \item[\textbf{(I}$\mathbf{1}$\textbf{)}] \underline{The $p$-Laplacian of $u^m$ (doubly nonlinear diffusion):} Let $p>1$ and $\alpha\geq1$. Choosing $a\equiv1, f\equiv \alpha^{p-1}$, and $m=(\alpha-1)(p-1)+1$ in $\eqref{main eq}$ yields
        \begin{equation}\label{plapm}
            \partial_t u \geq  \Delta_p(u^{\alpha}) + V(x,t) u^q \;\;\;\;\;\mathrm{in}\;M\times(0,\infty),
        \end{equation}
        where we recall that the $p$-Laplacian of a function $v$ is given by
        \begin{equation*}
            \Delta_p(v) = \mathrm{div}\Big( |\nabla v|^{p-2} \nabla v \Big).
        \end{equation*}
        %Indeed, we have
        %\begin{equation*}%\label{diffcoeff}
          %  \Delta_p(u^{\alpha})=\mathrm{div}\Big( \alpha^{p-1}u^{(\alpha-1)(p-1)} |\nabla u|^{p-2} \nabla u \Big)
        %\end{equation*}
        The differential inequalities that will be introduced below in \textbf{(I}$\mathbf{2}$\textbf{)} and \textbf{(I}$\mathbf{3}$\textbf{)} are special cases of $\eqref{plapm}$.
   \item[\textbf{(I}$\mathbf{2}$\textbf{)}] \underline{The $p$-Laplacian of $u$:} Inserting $m=1$ in $\eqref{plapm}$, we arrive at
        \begin{equation*}%\label{plap}
             \partial_t u \geq  \Delta_p(u) + V(x,t) u^q\;\;\;\;\;\mathrm{in}\;M\times(0,\infty).
        \end{equation*}
    \item[\textbf{(I}$\mathbf{3}$\textbf{)}] \underline{The Laplacian of $u^m$ (porous medium equation):} Letting $p=2$ in $\eqref{plapm}$, yields, for $m\geq1$,
        \begin{equation*}%\label{mezzpor}
            \partial_t u \geq  \Delta(u^m) + V(x,t) u^q\;\;\;\;\;\mathrm{in}\;M\times(0,\infty).
        \end{equation*}
   \item[\textbf{(I}$\mathbf{4}$\textbf{)}]
   \underline{(Generalized) evolution by mean curvature:} Setting
            \begin{equation*}
            f(|\nabla u|)\coloneq \frac{1}{(|\nabla u|^2+1)^{\theta}},
        \end{equation*}
        for some $\theta\geq 0, a\equiv1, m=1$, and $p=2$ in $\eqref{main eq}$, yields
        \begin{equation}\label{f}
            \partial_t u \geq  \mathrm{div}\left( \frac{\nabla u}{(|\nabla u|^2+1)^{\theta}} \right) + V(x,t) u^q\;\;\;\;\;\mathrm{in}\;M\times(0,\infty).
        \end{equation}
        Clearly, $0\leq f\leq1$, i.e. $f$ is indeed eligible in $\eqref{main eq}$. For $\theta=\frac{1}{2}$, this corresponds to the mean curvature equation for graphs with an additional potential term.
\end{itemize}
\begin{rmk}\label{thm:app1}
Assuming \textbf{HP}$\mathbf{1}$, respectively \textbf{HP}$\mathbf{2}$, one can identify the range of the parameter $q$ for which nonnegative, nontrivial solutions in the sense of Definition~\ref{thm:Defn} to $\eqref{plapm}-\eqref{f}$ above do not exist. This follows directly from Theorem~\ref{thm:Thm1}, respectively Theorem~\ref{thm:Thm2}, and the identification of the value $q_*\coloneq p+m-2$. Recall that both theorems require, as part of their assumptions, the condition $q>\max(q_*,1)$. The following table lists the value for $q_*$ for each of the above equations in $\eqref{plapm}-\eqref{f}$, and the range/value for the parameter $p$.

\begin{table}[ht]
\centering
\begin{tabular}{|p{1cm}|p{1.5cm}|p{1.5cm}|p{1.5cm}|p{1.5cm}|}
\hline
 & \textbf{(I}$\mathbf{1}$\textbf{)} & \textbf{(I}$\mathbf{2}$\textbf{)}, & \textbf{(I}$\mathbf{3}$\textbf{)} & \textbf{(I}$\mathbf{4}$\textbf{)} \\
\hline
$q_*$ & $\alpha(p-1)$ & $p-1$ & $m$ & $1$ \\
\hline
$p$ & $p>1$ & $p>1$ & $p=2$ & $p=2$ \\
\hline
\end{tabular}

\end{table}
\end{rmk}

We shall first state the Euclidean version of our results. Even in such case such results seem to be new in the present generality.

\begin{cor}\label{thm:cor1}
Let $(M,g)=(\mathbb{R}^{N},g_{\mathrm{flat}})$ with $N\geq 1$, $a\in \mathrm{Lip}_{\mathrm{loc}}(\mathbb{R}^{N})\cap L^{\infty}(\mathbb{R}^{N}), a>0$ a.e., and $V\equiv1$. Let further $p>1$ and $m\geq1$. If
\begin{equation*}%\label{cor1res}
    \max(1,p+m-2)<q\leq\frac{p}{N}+p+m-2,
\end{equation*}
then any global nonnegative weak solution to $\eqref{main eq}$ is trivial in $\mathbb{R}^N\times(0,\infty)$.
\end{cor}

\begin{rmk}%\label{thm:rmk2}
By Remark~\ref{thm:app1} and Corollary~\ref{thm:cor1}, we conclude for the problems in $\eqref{plapm}-\eqref{f}$: On $(\mathbb{R}^{N},g_{\mathrm{flat}})$ with $N\geq 1$, $a\in \mathrm{Lip}_{\mathrm{loc}}(\mathbb{R}^{N})\cap L^{\infty}(\mathbb{R}^{N}), a>0$ a.e., and $V\equiv1$, we have nonexistence of global nonnegative, nontrivial solutions in the following cases:
\begin{itemize}
    \item[\textbf{(I}$\mathbf{1}$\textbf{)}] if $\max(1,\alpha(p-1))<q\leq\frac{p}{N}+\alpha(p-1)$, so that, as a consequence, $p>\frac{N(\alpha+1)}{\alpha N+1}$.
   \item[\textbf{(I}$\mathbf{2}$\textbf{)}] if $\max(1,p-1)<q\leq\frac{p}{N}+p-1$, so that, as a consequence, $p>\frac{2N}{N+1}$.
     \item[\textbf{(I}$\mathbf{3}$\textbf{)}] if $m<q\leq\frac{2}{N}+m$.%, so that, as a consequence, $m>1-\frac{2}{n}$.
    \item[\textbf{(I}$\mathbf{4}$\textbf{)}] if $1<q\leq \frac{2}{N}+1$.
\end{itemize}
\end{rmk}

\begin{rmk}
As it is clear from the proof, Corollary~\ref{thm:cor1} holds on any complete and non-compact, $n$-dimensional, Riemannian manifold with $\mathrm{Vol}(B_R)\leq C R^{N}$ and all $R$ large enough. Note that $n$ and $N$ need not coincide.
\end{rmk}

We will now investigate the implications of Theorem~\ref{thm:Thm1} and Theorem~\ref{thm:Thm2} for more specific choices of the potential $V$ in $\eqref{main eq}$, namely
\begin{itemize}
    \item the case in which $V$ has separated variables, i.e., $V(x,t)\geq h(x)f(t)$, for some functions $h\colon M\to\mathbb{R}$ and $f\colon(0,\infty)\to\mathbb{R}$;
    \item the case in which $V$ is independent of time;
    \item the case $V\equiv1$.
\end{itemize}

%In the following corollaries, we explore Theorems~\ref{thm:Thm1} and~\ref{thm:Thm2} for potentials $V$ admitting a separation of spatial and temporal components.

\begin{cor}\label{thm:cor2}
Let $M$ be a complete, non-compact Riemannian manifold. Let further $p>1$, $m\geq1, q>\max(p+m-2,1)$, and $a\in\mathrm{Lip}_{\mathrm{loc}}(M), a>0$ a.e. Suppose that the potential $V\in L^1_{\mathrm{loc}}(S)$ satisfies
\begin{equation}\label{cor2vsep}
    V(x,t)\geq  h(x) f(t) \;\;\;\mathrm{for}\;\mathrm{a.e.}\;(x,t)\in M\times(0,\infty),
\end{equation}
for some functions $h\colon M\to\mathbb{R}$ and $f\colon(0,\infty)\to\mathbb{R}$. Further, suppose that $h$ and $f$ satisfy
\begin{equation}\label{cor2fhest}
\begin{split}
    0&\ < h(x)\leq C(1+r(x))^{\alpha_1} \;\;\;\mathrm{for}\;\mathrm{a.e.}\;x\in M,\\
     0&\ < f(t)\leq C(1+t)^{\alpha_2} \;\;\;\mathrm{for}\;\mathrm{a.e.}\;t\in(0,\infty), \\
\end{split}
\end{equation}
and
\begin{equation}\label{cor2hint}
    \int_{B_R}h(x) ^{-\frac{1}{q-1}}\;d\mu_a\leq C R^{\sigma_1}(\log R)^{\delta_1}, \;\;\;\;
     \int_0^T f(t)^{-\frac{1}{q-1}}\;dt\leq C T^{\sigma_2}(\log T)^{\delta_2},
\end{equation}
\begin{equation}\label{cor2fint}
    \int_{B_R}h(x) ^{-\frac{p+m-2}{q-p-m+2}}\;d\mu_a\leq C R^{\sigma_3}(\log R)^{\delta_3},
    \;\;\;\;\int_0^T f(t)^{-\frac{p+m-2}{q-p-m+2}}\;\leq C T^{\sigma_4}(\log T)^{\delta_4}
\end{equation}
for $T,R$ large enough, and with $\alpha_1,\alpha_2,\sigma_1,\sigma_2,\sigma_3,\sigma_4,\delta_1,\delta_2,\delta_3,\delta_4\geq0$ and $C>0$.

Then global nonnegative, nontrivial solutions to Problem $\eqref{main eq}$ do not exist  provided all the following conditions hold:
\begin{itemize}
    \item[(i)] $\delta_1+\delta_2<\frac{1}{q-1}$, $\delta_3+\delta_4<\frac{p+m-2}{q-p-m+2}$.
    \item[(ii)] $0\leq \sigma_2\leq\frac{q}{q-1}$, $0\leq\sigma_3\leq\frac{pq}{q-p-m+2}$.
    \item[(iii)] If $\sigma_2=\frac{q}{q-1}$, we require  $\sigma_1=0$. If $\sigma_3=\frac{pq}{q-p-m+2}$,  we require  $\sigma_4=0$.
    \item[(iv)] $\sigma_1\sigma_4\leq\left( \frac{q}{q-1}-\sigma_2 \right) \left( \frac{pq}{q-p-m+2}-\sigma_3 \right)$.
\end{itemize}
\end{cor}
\begin{cor}\label{thm:cor3}
Let $M$ be complete, non-compact Riemannian manifold. Let further $p>1,m\geq1,q>\max(p+m-2,1)$, and $a\in\mathrm{Lip}_{\mathrm{loc}}(M), a>0$ a.e. Suppose that the potential $V\in L^1_{\mathrm{loc}}(S)$ satisfies
\begin{equation*}%\label{cor3vsep}
    V(x,t)\geq  f(t) h(x) \;\;\;\mathrm{for}\;\mathrm{a.e.}\;(x,t)\in  M\times(0,\infty),
\end{equation*}
for some functions $h\colon M\to\mathbb{R}$ and $f\colon(0,\infty)\to\mathbb{R}$. Further, suppose that $h$ and $f$ satisfy
\begin{equation*}%\label{cor3fhest}
\begin{split}
    C^{-1}(1+r(x))^{-\alpha_1} &\ \leq h(x)\leq C(1+r(x))^{\alpha_1} \;\;\;\mathrm{for}\;\mathrm{a.e.}\;x\in M,\\
    C^{-1}(1+t)^{-\alpha_2} &\ \leq f(t)\leq C(1+t)^{\alpha_2} \;\;\;\mathrm{for}\;\mathrm{a.e.}\;t\in(0,\infty), \\
\end{split}
\end{equation*}
and
\begin{equation*}%%\label{cor3hint}
    \int_{B_R}h(x) ^{-\frac{1}{q-1}}\;d\mu_a\leq C R^{\sigma_1}(\log R)^{\delta_1}, \;\;\;\;  \int_0^T f(t)^{-\frac{1}{q-1}}\;dt\leq C T^{\sigma_2}(\log T)^{\delta_2},
\end{equation*}
\begin{equation*}%\label{cor3fint}
    \int_{B_R}h(x) ^{-\frac{p+m-2}{q-p-m+2}}\;d\mu_a\leq C R^{\sigma_3}(\log R)^{\delta_3},\;\;\;\;\int_0^T f(t)^{-\frac{p+m-2}{q-p-m+2}}\;dt\leq C T^{\sigma_4}(\log T)^{\delta_4}
\end{equation*}
for $T,R$ large enough, and with $\alpha_1,\alpha_2,\sigma_1,\sigma_2,\sigma_3,\sigma_4,\delta_1,\delta_2,\delta_3,\delta_4\geq0$ and $C>0$.

Then global nonnegative, nontrivial solutions to Problem $\eqref{main eq}$ do not exist provided all the following conditions hold:
\begin{itemize}
    \item[(i)] $\delta_1+\delta_2\leq\frac{1}{q-1}$, $\delta_3+\delta_4\leq\frac{p+m-2}{q-p-m+2}$.
    \item[(ii)] $0\leq \sigma_2\leq\frac{q}{q-1}$, $0\leq\sigma_3\leq\frac{pq}{q-p-m+2}$.
    \item[(iii)] If $\sigma_2=\frac{q}{q-1}$,  we require  $\sigma_1=0$. If $\sigma_3=\frac{pq}{q-p-m+2}$,  we require  $\sigma_4=0$.
    \item[(iv)] $\sigma_1\sigma_4\leq\left( \frac{q}{q-1}-\sigma_2 \right) \left( \frac{pq}{q-p-m+2}-\sigma_3 \right)$.
\end{itemize}
\end{cor}

\begin{rmk}
    Note in particular that the hypotheses in Corollary~\ref{thm:cor2} and Corollary~\ref{thm:cor3} allow for the potential $V$ to be independent of $x\in M$ or $t\in(0,\infty)$. We investigate the case of time-independent potentials below.

\end{rmk}
\begin{cor}\label{thm:cor2'}
Let $M$ be complete, non-compact Riemannian manifold. Let further $p>1,m\geq1,q>\max(p+m-2,1)$, and $a\in\mathrm{Lip}_{\mathrm{loc}}(M), a>0$ a.e. Suppose that the potential $V\in L^1_{\mathrm{loc}}(S)$ satisfies
\begin{equation*}%\label{cor2'vsep}
    V(x,t)\geq h(x) \;\;\;\mathrm{for}\;\mathrm{a.e.}\;(x,t)\in  M\times(0,\infty),
\end{equation*}
for some function $h\colon M\to\mathbb{R}$. Further, suppose that $h$ satisfies
\begin{equation*}%\label{cor2'hest}
    0 < h(x)\leq C(1+r(x))^{\alpha} \;\;\;\mathrm{for}\;\mathrm{a.e.}\;x\in M,
\end{equation*}
and
\begin{equation*}%\label{cor2'hint}
    \int_{B_R}h(x) ^{-\frac{1}{q-1}}\;d\mu_a\leq C R^{\sigma_1}(\log R)^{\delta_1}, \;\;\;\;  \int_{B_R}h(x) ^{-\frac{p+m-2}{q-p-m+2}}\;d\mu_a\leq C R^{\sigma_2}(\log R)^{\delta_2}
\end{equation*}
for $T,R$ large enough, and with $\alpha,\sigma_1,\sigma_2,\delta_1,\delta_2\geq0$ and $C>0$.

Then global nonnegative, nontrivial solutions to Problem $\eqref{main eq}$ do not exist provided all the following conditions hold:
    \begin{itemize}
        \item[(i)]  $\delta_1<\frac{1}{q-1}$, $\delta_2<\frac{p+m-2}{q-p-m+2}$,
    \item[(ii)]  $0\leq\sigma_2<\frac{pq}{q-p-m+2}$.
    \item[(iii)] $\sigma_1\leq \frac{1}{q-1} \left( \frac{pq}{q-p-m+2}-\sigma_2 \right)$.
\end{itemize}
\end{cor}

\begin{cor}\label{thm:cor3'}
Let $M$ be complete, non-compact Riemannian manifold. Let further $p>1,m\geq1,q>\max(p+m-2,1)$, and $a\in\mathrm{Lip}_{\mathrm{loc}}(M), a>0$ a.e. Suppose that the potential $V\in L^1_{\mathrm{loc}}(S)$ satisfies
\begin{equation*}%\label{cor3'vsep}
    V(x,t)\geq h(x) \;\;\;\mathrm{for}\;\mathrm{a.e.}\;(x,t)\in  M\times(0,\infty),
\end{equation*}
for some function $h\colon M\to\mathbb{R}$. Further, suppose that $h$ satisfies
\begin{equation*}%\label{cor3'hest}
    C^{-1}(1+r(x))^{-\alpha}  \leq h(x)\leq C(1+r(x))^{\alpha} \;\;\;\mathrm{for}\;\mathrm{a.e.}\;\;\;x\in M,
\end{equation*}
and
\begin{equation*}%\label{cor3'hint}
    \int_{B_R}h(x) ^{-\frac{1}{q-1}}\;d\mu_a\leq C R^{\sigma_1}(\log R)^{\frac{1}{q-1}}, \;\;\;\;  \int_{B_R}h(x) ^{-\frac{p+m-2}{q-p-m+2}}\;d\mu_a\leq C R^{\sigma_2}(\log R)^{\frac{p+m-2}{q-p-m+2}}
\end{equation*}
for $T,R$ large enough, and with $\alpha,\sigma_1,\sigma_2,\delta_1,\delta_2\geq0$ and $C>0$.

Then global nonnegative, nontrivial solutions to Problem $\eqref{main eq}$ do not exist provided all the following conditions hold:
    \begin{itemize}
    \item[(i)]  $0\leq\sigma_2<\frac{pq}{q-p-m+2}$.
    \item[(ii)] $\sigma_1\leq \frac{1}{q-1} \left( \frac{pq}{q-p-m+2}-\sigma_2 \right)$.
\end{itemize}
\end{cor}

Finally, we examine the case $V \equiv 1$ and establish nonexistence results under a specific volume growth condition on geodesic balls with respect to the weighted measure $\mu_a$.
\begin{cor}\label{thm:lastcor'}
Let $M$ be complete, non-compact Riemannian manifold. Let further $p>1,m\geq1,q>\max(p+m-2,1)$, and $a\in\mathrm{Lip}_{\mathrm{loc}}(M), a>0$ a.e. Suppose that the potential satisfies $V\equiv1$ on $S$. Assume that there exists some $C>0$ such that for all $R$ large enough
\begin{equation*}%\label{volgr}
    \mathrm{Vol}_{\mu_a}\left(B_R\right)\leq C R^{\frac{p}{q-p-m+2}} (\log R)^{\delta},
\end{equation*}
where $\delta=\min\left(\frac{1}{q-1},\frac{p+m-2}{q-p-m+2}\right)$. Then global nonnegative, nontrivial solutions to Problem $\eqref{main eq}$ do not exist.
\end{cor}
\begin{rmk}
    We explicitly observe that Corollary \ref{thm:lastcor'} yields nonexistence of global nonnegative, nontrivial solutions of Problem \eqref{plapm} when $V\equiv1$ if $p>1$, $\alpha\geq1$, $q>\max\{\alpha(p-1),1\}$ and $$\mathrm{Vol}\left(B_R\right)\leq C R^{\frac{p}{q-\alpha(p-1)}} (\log R)^{\delta},$$
where $\delta=\min\left(\frac{1}{q-1},\frac{\alpha(p-1)}{q-\alpha(p-1)}\right)$.
\end{rmk}

The rest of the paper is organized as follows. In Section 2 we provide the proofs of the main results, Theorems \ref{thm:Thm1} and ~\ref{thm:Thm2}, which rely on a careful test function argument and on a priori integral estimates. Section 3 is devoted to the proof of the ensuing corollaries.

\section{Proof of Theorems ~\ref{thm:Thm1} and ~\ref{thm:Thm2}}

\subsection{Proof of Theorem~\ref{thm:Thm1}}

We introduce the function $\chi_{\beta}\colon (0,\infty) \to [0,1]$ for $\beta>0$:
\begin{equation}\label{chib}
    t\mapsto\begin{cases}
       0 &\quad\text{if }t\leq\beta\\
   \frac{t}{\beta}-1 &\quad\text{if }\beta\leq t\leq 2\beta \\
    1 &\quad\text{if }2\beta\leq t.\\
     \end{cases}  \\
\end{equation}

To proceed with the proof, we need the following intermediate result.

\begin{lem}\label{thm:Lem1}
Let $p>1$, $m\geq1$ and $q>\max(p+m-2, 1)$. Let further \[s\geq \mathrm{max}\left(1, \frac{pq}{q-p-m+2},\frac{q}{q-1}\right)\] be fixed and let $\chi_{\beta}$ be as in $\eqref{chib}$. Let $u$ be a nonnegative weak solution $u$ of Problem $\eqref{main eq}$. Then there exists a constant $C>0$ (depending only on $p,q,m,K$ and $s$) such that, for all $\alpha \in \big(
-\frac{1}{2} \mathrm{min}\{1,p-1,p+m-2\},0 \big)$ and all $\phi\in \mathrm{Lip}(M\times[0,\infty)), 0\leq\phi\leq1$ with compact support, one has
\begin{equation}\label{Lem1res}
\begin{split}
    \int_0^{\infty}\int_M u^{q+\alpha} \phi^s &\ \chi_{\beta} V \;d\mu_a dt
    + %\frac{3}{4}
    |\alpha| \int_0^{\infty}\int_M |\nabla u|^p  u^{m+\alpha-2} \phi^s\chi_{\beta} f(|\nabla u|)  \;d\mu_a dt \\
    &\ \leq C \Bigg\{ |\alpha|^{-\frac{q(p-1)}{q-p-m+2}} \int_0^{\infty}\int_M   |\nabla\phi|^{p\frac{q+\alpha}{q-p-m+2}} V^{-\frac{p+m-2+\alpha}{q-p-m+2}} \;d\mu_a dt \\
    &\  \;\;\;\;\;\;\;\;\;\;\; +\int_0^{\infty}\int_M |\partial_t\phi|^{\frac{q+\alpha}{q-1}} V^{-\frac{\alpha+1}{q-1}} \;d\mu_a dt  \Bigg\}.
\end{split}
\end{equation}
Here and in the following, we adopt the convention that $\nabla u=0$ on level sets of $u$.
\begin{proof}
For any $\epsilon>0$, let $u_{\epsilon}\coloneq u+\epsilon$. Then $\psi\coloneq u_{\epsilon}^{\alpha}\phi^s\chi_{\beta}$ is bounded and inherits the regularity properties from $u$. We can approximate $\psi$ by a sequence of test functions in the sense of Definition ~\ref{thm:Defn}. By $\eqref{equivdefn}$ and a limiting argument, we have
\begin{equation*}
\begin{split}
    \int_0^{\infty}\int_M u^q &\ u_{\epsilon}^{\alpha}\phi^s\chi_{\beta} V \;d\mu_a dt \\
     \leq &\ \alpha \int_0^{\infty}\int_M |\nabla u|^p u_{\epsilon}^{\alpha-1}\phi^s\chi_{\beta} u^{m-1} f(|\nabla u|) \;d\mu_a dt \\
    &\  + s \int_0^{\infty}\int_M \langle \nabla \phi, \nabla u \rangle  u^{m-1} |\nabla u|^{p-2} u_{\epsilon}^{\alpha}\phi^{s-1}\chi_{\beta} f(|\nabla u|)\;d\mu_a dt\\
    &\  - \alpha \int_0^{\infty}\int_M  u_{\epsilon}^{\alpha-1}\phi^s\chi_{\beta} (\partial_t u)u \;d\mu_a dt \\
    &\  - s \int_0^{\infty}\int_M u_{\epsilon}^{\alpha}\phi^{s-1}\chi_{\beta} (\partial_t \phi)u \;d\mu_a dt \\
   &\  - \int_0^{\infty}\int_M u_{\epsilon}^{\alpha}\phi^s\chi_{\beta}' u \;d\mu_a dt. \\
\end{split}
\end{equation*}
Since all terms in the first integral on the right-hand side are positive and $\alpha$ is negative, this is equivalent to
\begin{equation}\label{Lem11}
\begin{split}
    \int_0^{\infty}\int_M u^q &\ u_{\epsilon}^{\alpha}\phi^s\chi_{\beta} V \;d\mu_a dt
    + |\alpha| \int_0^{\infty}\int_M |\nabla u|^p u_{\epsilon}^{\alpha-1}\phi^s\chi_{\beta} u^{m-1} f(|\nabla u|) \;d\mu_a dt \\
    &\ \leq s \int_0^{\infty}\int_M \langle \nabla \phi, \nabla u \rangle  u^{m-1} |\nabla u|^{p-2} u_{\epsilon}^{\alpha}\phi^{s-1}\chi_{\beta} f(|\nabla u|)\;d\mu_a dt\\
    &\ \;\;\;\; + I,
\end{split}
\end{equation}
where we set
\begin{equation}\label{I}
\begin{split}
    I  \coloneq &\ -\alpha \int_0^{\infty}\int_M  u_{\epsilon}^{\alpha-1}\phi^s\chi_{\beta} (\partial_t u)u \;d\mu_a dt
    - s \int_0^{\infty}\int_M u_{\epsilon}^{\alpha}\phi^{s-1}\chi_{\beta} (\partial_t \phi)u \;d\mu_a dt \\
     &\ - \int_0^{\infty}\int_M u_{\epsilon}^{\alpha}\phi^s\chi_{\beta}' u \;d\mu_a dt.\\
\end{split}
\end{equation}
Using $u=u_{\epsilon}-\epsilon$, we can rewrite the first integral of $I$ as follows:
\begin{equation*}
\begin{split}
   - \alpha \int_0^{\infty}\int_M  u_{\epsilon}^{\alpha-1}\phi^s\chi_{\beta} (\partial_t u)u \;d\mu_a dt
     = &\ -\alpha \int_0^{\infty}\int_M u_{\epsilon}^{\alpha}\phi^s\chi_{\beta} (\partial_t u) \;d\mu_a dt \\
    &\ +\alpha \epsilon \int_0^{\infty}\int_M u_{\epsilon}^{\alpha-1}\phi^s\chi_{\beta} (\partial_t u) \;d\mu_a dt \\
    = &\ -\frac{\alpha}{\alpha+1} \int_0^{\infty}\int_M \partial_t (u_{\epsilon}^{\alpha+1})\phi^s\chi_{\beta}  \;d\mu_a dt \\
    &\ + \epsilon \int_0^{\infty}\int_M \partial_t (u_{\epsilon}^{\alpha})\phi^s\chi_{\beta}  \;d\mu_a dt. \\
\end{split}
\end{equation*}
Integrating by parts, yields
\begin{equation*}
\begin{split}
    - \alpha \int_0^{\infty}\int_M  u_{\epsilon}^{\alpha-1}\phi^s\chi_{\beta} (\partial_t u)u \;d\mu_a dt
   = &\ \frac{\alpha}{\alpha+1} \int_0^{\infty}\int_M u_{\epsilon}^{\alpha+1}\partial_t(\phi^s\chi_{\beta} ) \;d\mu_a dt \\
    &\ - \epsilon \int_0^{\infty}\int_M u_{\epsilon}^{\alpha}\partial_t(\phi^s\chi_{\beta}) \;d\mu_a dt \\
   = &\ s \frac{\alpha}{\alpha+1} \int_0^{\infty}\int_M u_{\epsilon}^{\alpha+1}\phi^{s-1}(\partial_t\phi)\chi_{\beta} \;d\mu_a dt \\
   &\ + \frac{\alpha}{\alpha+1} \int_0^{\infty}\int_M u_{\epsilon}^{\alpha+1}\phi^s\chi_{\beta}' \;d\mu_a dt \\
    &\ - s\epsilon \int_0^{\infty}\int_M u_{\epsilon}^{\alpha}\phi^{s-1}(\partial_t\phi)\chi_{\beta} \;d\mu_a dt \\
    &\ - \epsilon \int_0^{\infty}\int_M u_{\epsilon}^{\alpha}\phi^s\chi_{\beta}' \;d\mu_a dt \\
\end{split}
\end{equation*}
Thus, again using $u=u_{\epsilon}-\epsilon$, in the last two terms of $I$ in $\eqref{I}$, $I$ becomes
\begin{equation*}
\begin{split}
    I = &\ -\frac{s}{\alpha+1} \int_0^{\infty}\int_M u_{\epsilon}^{\alpha+1} \phi^{s-1} (\partial_t\phi) \chi_{\beta}\;d\mu_a dt \\
    &\ -\frac{1}{\alpha+1} \int_0^{\infty}\int_M u_{\epsilon}^{\alpha+1} \phi^s \chi_{\beta}'\;d\mu_a dt. \\
\end{split}
\end{equation*}
Inserting this into $\eqref{Lem11}$, we arrive at
\begin{equation*}
\begin{split}
    \int_0^{\infty}\int_M u^q &\ u_{\epsilon}^{\alpha}\phi^s\chi_{\beta} V \;d\mu_a dt
    + |\alpha| \int_0^{\infty}\int_M |\nabla u|^p u_{\epsilon}^{\alpha-1}\phi^s\chi_{\beta} u^{m-1} f(|\nabla u|) \;d\mu_a dt \\
   &\ \;\;\;\; + \frac{1}{\alpha+1} \int_0^{\infty}\int_M u_{\epsilon}^{\alpha+1} \phi^s \chi_{\beta}'\;d\mu_a dt\\
    &\ \leq s \int_0^{\infty}\int_M \langle \nabla \phi, \nabla u \rangle  u^{m-1} |\nabla u|^{p-2} u_{\epsilon}^{\alpha}\phi^{s-1}\chi_{\beta} f(|\nabla u|)\;d\mu_a dt\\
    &\ \;\;\;\; -\frac{s}{\alpha+1} \int_0^{\infty}\int_M u_{\epsilon}^{\alpha+1} \phi^{s-1} (\partial_t\phi) \chi_{\beta}\;d\mu_a dt. \\
\end{split}
\end{equation*}
In particular, since the last term on the left-hand side is nonnegative for all $\alpha>-\frac{1}{2}$,
\begin{equation}\label{Lem12}
\begin{split}
    \int_0^{\infty}\int_M u^q &\ u_{\epsilon}^{\alpha}\phi^s\chi_{\beta} V \;d\mu_a dt
    + |\alpha| \int_0^{\infty}\int_M |\nabla u|^p u_{\epsilon}^{\alpha-1}\phi^s\chi_{\beta} u^{m-1} f(|\nabla u|) \;d\mu_a dt \\
    &\ \leq s \int_0^{\infty}\int_M \langle \nabla \phi, \nabla u \rangle  u^{m-1} |\nabla u|^{p-2} u_{\epsilon}^{\alpha}\phi^{s-1}\chi_{\beta} f(|\nabla u|)\;d\mu_a dt\\
    &\ \;\;\;\; -\frac{s}{\alpha+1} \int_0^{\infty}\int_M u_{\epsilon}^{\alpha+1} \phi^{s-1} (\partial_t\phi) \chi_{\beta}\;d\mu_a dt. \\
\end{split}
\end{equation}
Now we apply Young's Inequality with $\frac{p}{p-1}$ and $p$ to the first term on the right-hand side in $\eqref{Lem12}$ as follows:
\begin{equation*}
    \begin{split}
         s \int_0^{\infty}\int_M &\ \langle \nabla \phi, \nabla u \rangle  u^{m-1} |\nabla u|^{p-2} u_{\epsilon}^{\alpha}\phi^{s-1}\chi_{\beta} f(|\nabla u|)\;d\mu_a dt\\
         &\ \leq  s \int_0^{\infty}\int_M   |\nabla \phi| |\nabla u|^{p-1}   u^{m-1}  u_{\epsilon}^{\alpha}\phi^{s-1}\chi_{\beta} f(|\nabla u|)\;d\mu_a dt\\
         &\ =  \int_0^{\infty}\int_M \Bigg\{ \Big( \frac{p|\alpha|}{4(p-1)} \Big)^{\frac{p-1}{p}} |\nabla u|^{p-1} u_{\epsilon}^{(\alpha-1)\frac{p-1}{p}} u^{(m-1)\frac{p-1}{p}}\phi^{s\frac{p-1}{p}} \chi_{\beta}^{\frac{p-1}{p}} f(|\nabla u|)^{\frac{p-1}{p}} \Bigg\}\\
         &\ \;\;\;\; \times \Bigg\{ s \Big( \frac{p|\alpha|}{4(p-1)} \Big)^{-\frac{p-1}{p}} |\nabla\phi| u_{\epsilon}^{1+\frac{\alpha-1}{p}}u^{(m-1)\frac{1}{p}}\phi^{\frac{s}{p}-1} \chi_{\beta}^{\frac{1}{p}} f(|\nabla u|)^{\frac{1}{p}}  \Bigg\} \;d\mu_a dt \\
         &\ \leq \frac{|\alpha|}{4} \int_0^{\infty}\int_M |\nabla u|^p u_{\epsilon}^{\alpha-1}\phi^s\chi_{\beta} u^{m-1} f(|\nabla u|) \;d\mu_a dt \\
         &\ \;\;\;\; +  K  \frac{s}{p} \Big( \frac{4s(p-1)}{p|\alpha|} \Big)^{p-1} \int_0^{\infty}\int_M |\nabla\phi|^pu_{\epsilon}^{p+\alpha-1}u^{m-1}\phi^{s-p} \chi_{\beta}  \;d\mu_a dt, \\
    \end{split}
\end{equation*}
where $K$ is as in the statements of Theorems \ref{thm:Thm1}, \ref{thm:Thm2}, so that $0\le f\le K$.  Since
\begin{equation*}
    K\frac{s}{p} \Big( \frac{4s(p-1)}{p|\alpha|} \Big)^{p-1} \leq C |\alpha|^{-(p-1)},
\end{equation*}
where $C$ is a constant depending on $s, K, p$.   We see from Inequality $\eqref{Lem12}$ that
\begin{equation}\label{Lem13}
\begin{split}
    \int_0^{\infty}\int_M u^q &\ u_{\epsilon}^{\alpha}\phi^s\chi_{\beta} V \;d\mu_a dt
    + \frac{3}{4}|\alpha| \int_0^{\infty}\int_M |\nabla u|^p u_{\epsilon}^{\alpha-1}\phi^s\chi_{\beta} u^{m-1} f(|\nabla u|) \;d\mu_a dt \\
    &\ \leq C |\alpha|^{-(p-1)} \int_0^{\infty}\int_M
|\nabla\phi|^p u_{\epsilon}^{p+\alpha-1}u^{m-1}\phi^{s-p} \chi_{\beta}  \;d\mu_a dt \\
    &\ -\frac{s}{\alpha+1} \int_0^{\infty}\int_M u_{\epsilon}^{\alpha+1} \phi^{s-1} (\partial_t\phi) \chi_{\beta}\;d\mu_a dt. \\
\end{split}
\end{equation}
In order to estimate the first integral on the right-hand side in $\eqref{Lem13}$, we make use of Young's Inequality again. This time with the following exponents:
\begin{equation}\label{Lem1b}
   b\coloneq \frac{q+\alpha}{p+\alpha+m-2}\;\;\;\;\;\mathrm{and}\;\;\;\;\;b'=\frac{b}{b-1}=\frac{q+\alpha}{q-p-m+2}.
\end{equation}
Note that this is well-defined for $q>p+m-2$ and $\alpha > -\frac{1}{2}(p+m-2)$. We have
\begin{equation}\label{Lem14}
\begin{split}
    C |\alpha|^{-(p-1)} &\ \int_0^{\infty}\int_M  |\nabla\phi|^p u_{\epsilon}^{p+\alpha-1}u^{m-1}\phi^{s-p} \chi_{\beta}   \;d\mu_a dt \\
   &\ = \int_0^{\infty}\int_M \Bigg\{ \Big( \frac{b}{4} \Big)^{\frac{1}{b}} u_{\epsilon}^{p+\alpha-1} u^{m-1} \phi^{\frac{s}{b}} \chi_{\beta}^{\frac{1}{b}} V^{\frac{1}{b}}\Bigg\} \\
    &\ \;\;\;\; \times \Bigg\{ C \Big( \frac{b}{4} \Big)^{-\frac{1}{b}}|\alpha|^{-(p-1)}
|\nabla\phi|^p \phi^{\frac{s}{b'}-p} \chi_{\beta}^{\frac{1}{b'}}    V^{-\frac{1}{b}} \Bigg\} \;d\mu_a dt \\
    &\ \leq \frac{1}{4} \int_0^{\infty}\int_M u_{\epsilon}^{(p+\alpha-1)b} u^{(m-1)b} \phi^s\chi_{\beta} V  \;d\mu_a dt \\
    &\ \;\;\;\; + \frac{1}{b'}C^{b'}  \Big( \frac{b}{4} \Big)^{-\frac{b'}{b}} |\alpha|^{-(p-1)b'} \int_0^{\infty}\int_M |\nabla\phi|^{pb'} \phi^{s-pb'} \chi_{\beta}  V^{-\frac{b'}{b}}  \;d\mu_a dt. \\
\end{split}
\end{equation}
Observe that, by the assumptions on $\alpha$ in the statement of this Lemma,
\begin{equation*}
    \frac{1}{b'}C^{b'} \Big( \frac{b}{4} \Big)^{-\frac{b'}{b}} |\alpha|^{-(p-1)b'} \leq C|\alpha|^{-\frac{(p-1)(q+\alpha)}{q-p-m+2}}\leq C |\alpha|^{-\frac{(p-1)q}{q-p-m+2}},
\end{equation*}
where $C$ is another suitable constant depending on $K, s, p, m, q$.
Inserting this and the values for $b$ and $b'$ in $\eqref{Lem1b}$ into $\eqref{Lem14}$, and the resulting estimate into $\eqref{Lem13}$, yields
\begin{equation}\label{Lem15}
\begin{split}
   \int_0^{\infty} &\ \int_M u^q  u_{\epsilon}^{\alpha}\phi^s\chi_{\beta} V \;d\mu_a dt
   -  \frac{1}{4} \int_0^{\infty}\int_M u_{\epsilon}^{\frac{(p+\alpha-1)(q+\alpha)}{p+\alpha+m-2}}
   u^{\frac{(m-1)(q+\alpha)}{p+\alpha+m-2}} \phi^s \chi_{\beta} V  \;d\mu_a dt \\
  &\ \;\;\;\; + \frac{3}{4}|\alpha| \int_0^{\infty}\int_M |\nabla u|^p u_{\epsilon}^{\alpha-1}\phi^s\chi_{\beta} u^{m-1} f(|\nabla u|) \;d\mu_a dt\\
   &\ \leq C |\alpha|^{-\frac{(p-1)q}{q-p-m+2}}
   \int_0^{\infty}\int_M |\nabla\phi|^{p\frac{q+\alpha}{q-p-m+2}} \phi^{s-p\frac{q+\alpha}{q-p-m+2}} \chi_{\beta}
    V^{-\frac{p+\alpha+m-2}{q-p-m+2}}  \;d\mu_a dt \\
   &\ \;\;\;\; -\frac{s}{\alpha+1} \int_0^{\infty}\int_M u_{\epsilon}^{\alpha+1} \phi^{s-1} (\partial_t\phi)
   \chi_{\beta} \;d\mu_a dt. \\
\end{split}
\end{equation}
We apply Young's Inequality once more to the second integral on the left-hand side in $\eqref{Lem15}$ with
\begin{equation}\label{Lem1c}
    c\coloneq\frac{q+\alpha}{\alpha+1}\;\;\;\;\;\mathrm{and}\;\;\;\;\;c'=\frac{c}{c-1}=\frac{q+\alpha}{q-1}.
\end{equation}
These exponents are well-defined since $q>1$ and $\alpha>-1/2$. Then
\begin{equation}\label{Lem16}
\begin{split}
    -\frac{s}{\alpha+1} \int_0^{\infty} &\ \int_M   u_{\epsilon}^{\alpha+1} \phi^{s-1}
    (\partial_t \phi) \chi_{\beta} \;d\mu_a dt \\
    &\ \leq  \int_0^{\infty}\int_M \Bigg\{ \Big( \frac{c}{4} \Big)^{\frac{1}{c}}
    u_{\epsilon}^{\alpha+1} \phi^{\frac{s}{c}} \chi_{\beta}^{\frac{1}{c}} V^{\frac{1}{c}} \Bigg\}\\
    &\ \quad \;\;\;\;\;\;\;\;\;\;\;\;\;\;\;  \times \Bigg\{ \frac{s}{\alpha+1}\Big( \frac{c}{4} \Big)^{-\frac{1}{c}}
    \phi^{\frac{s}{c'} - 1} \chi_{\beta}^{\frac{1}{c'}}  |\partial_t \phi| V^{-\frac{1}{c}}  \Bigg\}\;d\mu_a dt\\
    &\ \leq \frac{1}{4} \int_0^{\infty}\int_M u_{\epsilon}^{q+\alpha} \phi^s
    \chi_{\beta} V \;d\mu_a dt \\
    &\ \quad + \frac{1}{c'} \Big( \frac{s}{\alpha+1} \Big)^{c'}
    \Big( \frac{c}{4}  \Big)^{-\frac{c'}{c}}
    \int_0^{\infty}\int_M \phi^{s-c'} \chi_{\beta}  |\partial_t \phi|^{c'} V^{-\frac{c'}{c}} \;d\mu_a dt. \\
\end{split}
\end{equation}
By the running assumptions on $\alpha$, it follows that $\alpha>-1/2$. This is enough to guarantee that $c, c'$ are bounded and bounded away from zero, for any fixed $q$. Hence, for a suitable constant $C$: %Observe that, by the assumptions on $\alpha$ in the statement of this Lemma,
\begin{equation}\label{Lem1cest}
    \frac{1}{c'}\Big(\frac{s}{\alpha+1}\Big)^{c'}\Big( \frac{1}{4} c \Big)^{-\frac{c'}{c}} \leq C.
\end{equation}
Inserting $\eqref{Lem16}, \eqref{Lem1cest}$ and definition of $c, c'$ in $\eqref{Lem1c}$ into $\eqref{Lem15}$ yields
\begin{equation*}
\begin{split}
    \int_0^{\infty}\int_M u^q &\ u_{\epsilon}^{\alpha} \phi^s \chi_{\beta} V \;d\mu_a dt  -  \frac{1}{4} \int_0^{\infty}\int_M u_{\epsilon}^{\frac{(p+\alpha-1)(q+\alpha)}{p+\alpha+m-2}}
    u^{\frac{(m-1)(q+\alpha)}{p+\alpha+m-2}} \phi^s \chi_{\beta} V \;d\mu_a dt \\
    &\ \;\;\;\; - \frac{1}{4} \int_0^{\infty}\int_M u_{\epsilon}^{q+\alpha}
    \phi^s \chi_{\beta} V \;d\mu_a dt \\
    &\ \;\;\;\; + \frac{3}{4} |\alpha| \int_0^{\infty}\int_M |\nabla u|^p u_{\epsilon}^{\alpha-1}
    \phi^s \chi_{\beta} u^{m-1} f(|\nabla u|) \;d\mu_a dt \\
    &\ \leq C |\alpha|^{-\frac{(p-1)q}{q-p-m+2}} \\
    &\ \;\;\;\; \times\int_0^{\infty}\int_M |\nabla \phi|^{p \frac{q+\alpha}{q-p-m+2}} \phi^{s-p\frac{q+\alpha}{q-p-m+2}} \chi_{\beta} V^{-\frac{p+\alpha+m-2}{q-p-m+2}} \;d\mu_a dt \\
    &\ \;\;\;\; + C \int_0^{\infty}\int_M \phi^{s-\frac{q+\alpha}{q-1}}
    \chi_{\beta} |\partial_t \phi|^{\frac{q+\alpha}{q-1}} V^{-\frac{\alpha+1}{q-1}} \;d\mu_a dt. \\
\end{split}
\end{equation*}
In particular, by the choice $s\geq \mathrm{max}\big(1, \frac{pq}{q-p-m+2},\frac{q}{q-1}\big)$, and since $0\leq\phi,\chi_{\beta}\leq1$, this implies that
\begin{equation}\label{Lem1ffin}
\begin{split}
    \int_0^{\infty}\int_M u^q &\ u_{\epsilon}^{\alpha} \phi^s \chi_{\beta} V \;d\mu_a dt
    - \frac{1}{4} \int_0^{\infty}\int_M u_{\epsilon}^{\frac{(p+\alpha-1)(q+\alpha)}{p+\alpha+m-2}}
    u^{\frac{(m-1)(q+\alpha)}{p+\alpha+m-2}} \phi^s \chi_{\beta} V \;d\mu_a dt \\
    &\ \quad- \frac{1}{4} \int_0^{\infty}\int_M u_{\epsilon}^{q+\alpha}
    \phi^s \chi_{\beta} V \;d\mu_a dt \\
    &\ \quad + \frac{3}{4} |\alpha| \int_0^{\infty}\int_M |\nabla u|^p u_{\epsilon}^{\alpha-1}
    \phi^s \chi_{\beta} u^{m-1} f(|\nabla u|) \;d\mu_a dt \\
    &\ \leq C |\alpha|^{-\frac{(p-1)q}{q-p-m+2}}
    \int_0^{\infty}\int_M
    |\nabla \phi|^{p \frac{q+\alpha}{q-p-m+2}} V^{-\frac{p+\alpha+m-2}{q-p-m+2}} \;d\mu_a dt \\
    &\ \quad + C \int_0^{\infty}\int_M |\partial_t \phi|^{\frac{q+\alpha}{q-1}}
    V^{-\frac{\alpha+1}{q-1}} \;d\mu_a dt. \\
\end{split}
\end{equation}
Finally, we let $\epsilon \rightarrow 0$ and use Fatou's Lemma, which yields %$\eqref{Lem1res}$ as claimed:
\begin{equation*}
\begin{split}
   \frac{1}{2} \int_0^{\infty}&\ \int_M   u^{q+\alpha} \phi^s\chi_{\beta} V \;d\mu_a dt + \frac{3}{4}|\alpha| \int_0^{\infty}\int_M |\nabla u|^p u^{m+\alpha-2}\phi^s\chi_{\beta} f(|\nabla u|) \;d\mu_a dt \\
    &\ \leq C \Bigg\{ |\alpha|^{-\frac{(p-1)q}{q-p-m+2}} \int_0^{\infty}\int_M |\nabla\phi|^{p\frac{q+\alpha}{q-p-m+2}} V^{-\frac{p+\alpha+m-2}{q-p-m+2}}  \;d\mu_a dt \\
&\ \;\;\;\;\;\;\;\;\;\;\;\;+\int_0^{\infty}\int_M |\partial_t\phi|^{\frac{q+\alpha}{q-1}} V^{-\frac{\alpha+1}{q-1}}  \;d\mu_a dt\Bigg\}, \\
\end{split}
\end{equation*}
\noindent and hence \eqref{Lem1res}. Here we used the convention that $\nabla u=0$ on level sets of $u$. The convergence of the first and last integral on the left-hand side in $\eqref{Lem1ffin}$ can be proven by Beppo-Levi since $u_{\epsilon}^{\alpha}$ and $u_{\epsilon}^{\alpha-1}$ are monotonously increasing as $\epsilon$ goes to zero; recall that $\alpha < 0$. The powers of $u_{\epsilon}$ in the two remaining integrals on the left-hand side are positive by the assumptions on $\alpha$ in the statement of this Lemma. Thus, in order to apply the Dominated Convergence Theorem, noting that in both cases the integrands are nonnegative, and decreasing in $\epsilon$,
we need to show that, for some $\epsilon >0$,
\begin{equation}\label{Lem1MC1}
     \int_0^{\infty}\int_M u_{\epsilon}^{\frac{(p+\alpha-1)(q+\alpha)}{p+\alpha+m-2}} u^{\frac{(m-1)(q+\alpha)}{p+\alpha+m-2}} \phi^s\chi_{\beta} V  \;d\mu_a dt < \infty
\end{equation}
and
\begin{equation}\label{Lem1MC2}
      \int_0^{\infty}\int_M    u_{\epsilon}^{q+\alpha}\phi^s\chi_{\beta} V \;d\mu_a dt < \infty.
\end{equation}

Inequality $\eqref{Lem1MC2}$ can be proven by the assumption that $u\in L_{\mathrm{loc}}^{q}(M\times(0,\infty);V\;d\mu_a dt)$ since $\phi$ has compact support. Inequality $\eqref{Lem1MC1}$ follows from $\eqref{Lem1MC2}$. This completes the proof of Lemma \ref{thm:Lem1}.
\end{proof}
\end{lem}
We now turn to the proof of \textbf{Theorem~\ref{thm:Thm1}}:
\begin{proof}
We will show that, for fixed $\beta > 0$,
\begin{equation}\label{thm11}
     \int_0^{\infty} \int_M   u^q \chi_{\beta} V \;d\mu_a dt  = 0.
\end{equation}
Then, taking the limit $\beta\rightarrow 0$, we have, by Fatou's Lemma and by the positivity of all functions involved, that $u=0$ a.e. on $M\times(0,\infty)$.

Firstly, observe that Hypothesis \textbf{HP}$\mathbf{1}$ provides, in particular, growth estimates for sets of the form $E_{2^{1/\theta_2}nR}\setminus E_{nR}$, with $\theta_2\geq1$, which will be the relevant sets appearing in the proof below.

In the following, let $C_0, \theta_1, \theta_2$ be as in \textbf{HP}$\mathbf{1}$; let $\alpha=-\frac{1}{\log R}$ and $C_1 > \frac{C_0+\theta_2+1}{\theta_2}$. In order to show Equality $\eqref{thm11}$, inspired by an idea of \cite{GrigS} for the elliptic, semilinear case, we use Lemma~\ref{thm:Lem1} and insert the sequence $(\phi_n)_{n\in\mathbb{N}}$ of test functions into $\eqref{Lem1res}$, where $\phi_n\coloneq\phi\eta_n$, with
\begin{equation}\label{phi}
\phi(x,t)\coloneq\begin{cases}
       1 &\quad\text{if }(x,t)\in E_R\\
   \Big(\frac{r(x)^{\theta_2}+t^{\theta_1}}{R^{\theta_2}}\Big)^{C_1\alpha}&\quad\text{if }(x,t)\in E_R^c, \\
   \end{cases}
\end{equation}
\noindent and for all $n\in\mathbb{N}$,
\begin{equation}\label{eta}
\eta_n(x,t)\coloneq\begin{cases}
       1 &\quad\text{if }(x,t)\in E_{nR}\\
   2- \frac{r(x)^{\theta_2}+t^{\theta_1}}{(nR)^{\theta_2}}&\quad\text{if }(x,t)\in E_{2^{1/\theta_2}nR}\setminus E_{nR} \\
    0 &\quad\text{if }(x,t)\in E_{2^{1/\theta_2}nR}^c.\\
     \end{cases}  \\
\end{equation}
Note that $\phi_n\in\mathrm{Lip}(S)$ with $0\leq\phi_n\leq1$ and that
\begin{equation*}
    \partial_t\phi_n=(\partial_t\phi)\eta_n+\phi(\partial_t\eta_n),\;\;\;\;\;\nabla\phi_n=(\nabla\phi)\eta_n+\phi(\nabla\eta_n)
\end{equation*}

\noindent a.e. in $S$. In addition, we have for every $a\geq1$,
\begin{equation}\label{a}
     |\partial_t\phi_n|^a\leq 2^{a-1}(|\partial_t\phi|^a+\phi^a|\partial_t\eta_n|^a),\;\;\;\;\;     |\nabla\phi_n|^a\leq 2^{a-1}(|\nabla\phi|^a+\phi^a|\eta_n|^a).
\end{equation}
Inserting $\phi_n$ into $\eqref{Lem1res}$, yields, with $s\geq \mathrm{max}\big(1, \frac{pq}{q-p-m+2},\frac{q}{q-1}\big)$, and $|\alpha|$ small enough:
\begin{equation}\label{thm1maineq}
\begin{split}
   \int_0^{\infty} \int_M &\  u^{q+\alpha}  \phi_n^s\chi_{\beta} V \;d\mu_a dt \\
    &\ \leq C \Bigg\{ |\alpha|^{-\frac{(p-1)q}{q-p-m+2}} \int_0^{\infty}\int_M |\nabla\phi_n|^{\frac{p(q+\alpha)}{q-p-m+2}} V^{-\frac{p+\alpha+m-2}{q-p-m+2}}  \;d\mu_a dt \\
&\ \;\;\;\;\;\;\;\;\;\;\;\; +\int_0^{\infty}\int_M |\partial_t\phi_n|^{\frac{q+\alpha}{q-1}} V^{-\frac{\alpha+1}{q-1}}  \;d\mu_a dt\Bigg\} \\
&\ \leq C \Bigg\{ |\alpha|^{-\frac{(p-1)q}{q-p-m+2}} \Bigg(\int\int_{E_R^c} |\nabla\phi|^{\frac{p(q+\alpha)}{q-p-m+2}} V^{-\frac{p+\alpha+m-2}{q-p-m+2}}  \;d\mu_a dt \\
&\ \;\;\;\;\;\;\;\;\;\;\;\;+ \int\int_{E_{2^{1/\theta_2}nR}\setminus E_{nR}}\phi^{\frac{p(q+\alpha)}{q-p-m+2}}|\nabla\eta_n|^{\frac{p(q+\alpha)}{q-p-m+2}}  V^{-\frac{p+\alpha+m-2}{q-p-m+2}}  \;d\mu_a dt \Bigg) \\
&\ \;\;\;\;\;\;\;\;\;\;\;\; +\int\int_{E_R^c} |\partial_t\phi|^{\frac{q+\alpha}{q-1}} V^{-\frac{\alpha+1}{q-1}}  \;d\mu_a dt \\
&\ \;\;\;\;\;\;\;\;\;\;\;\; +\int\int_{E_{2^{1/\theta_2}nR}\setminus E_{nR}}\phi^{\frac{q+\alpha}{q-1}}|\partial_t\eta_n|^{\frac{q+\alpha}{q-1}} V^{-\frac{\alpha+1}{q-1}}  \;d\mu_a dt\Bigg\} \\
&\ = C\Big\{ |\alpha|^{-\frac{(p-1)q}{q-p-m+2}} (I_1+I_2)+I_3+I_4 \Big\},
\end{split}
\end{equation}
where
\begin{align}
    I_1 & \coloneq \int\int_{E_R^c}|\nabla\phi|^{\frac{p(q+\alpha)}{q-p-m+2}} V^{-\frac{p+\alpha+m-2}{q-p-m+2}}  \;d\mu_a dt, \label{I_1}\\
    I_2 & \coloneq\int\int_{E_{2^{1/\theta_2}nR}\setminus E_{nR}}\phi^{\frac{p(q+\alpha)}{q-p-m+2}}|\nabla\eta_n|^{\frac{p(q+\alpha)}{q-p-m+2}}  V^{-\frac{p+\alpha+m-2}{q-p-m+2}}  \;d\mu_a dt, \label{I_2}\\
    I_3 & \coloneq\int\int_{E_R^c} |\partial_t\phi|^{\frac{q+\alpha}{q-1}} V^{-\frac{\alpha+1}{q-1}}  \;d\mu_a dt, \label{I_3}\\
    I_4 & \coloneq\int\int_{E_{2^{1/\theta_2}nR}\setminus E_{nR}}\phi^{\frac{q+\alpha}{q-1}}|\partial_t\eta_n|^{\frac{q+\alpha}{q-1}} V^{-\frac{\alpha+1}{q-1}}  \;d\mu_a dt. \label{I_4}
\end{align}

We will start by estimating $I_3$ and $I_4$. These two integrals coincide with the integrals $I_3$ and $I_4$ in the proof of Theorem $2$ in \cite{dariospaper}, with the only exception that we integrate against the weighted measure $\mu_a$ in $\eqref{I_3}$ and $\eqref{I_4}$. Proceeding precisely as in \cite{dariospaper}, one sees that $\eqref{HP11}$ in Hypothesis \textbf{HP}$\mathbf{1}$ leads to the following estimate for $R$ large enough:
\begin{equation}\label{I34fin}
    I_3+I_4\leq C\Big( |\alpha|^{\frac{1}{q-1}-s_2}+ n^{-\frac{|\alpha|}{q-1}} \big[ \log(nR)\big]^{s_2} \Big).
\end{equation}
Indeed, Estimate $(1.6)$ in \cite{dariospaper} agrees with $\eqref{HP11}$ in \textbf{HP}$\mathbf{1}$ above, except for the weighted measure $\mu_a$.

In order to estimate $I_1$, we compute $\nabla \phi$ for $\phi$ in $\eqref{phi}$ and use the fact that $|\nabla r(x)|\leq 1$ for a.e. all $x\in M$. This yields
\begin{equation}\label{I11}
\begin{split}
    &\ |\alpha|^{-\frac{(p-1)q}{q-p-m+2}}  I_1 \\
    &\leq C\ |\alpha|^{-\frac{(p-1)q}{q-p-m+2}} \int\int_{E_R^c}
    \Bigg[ C_1|\alpha|\theta_2 \Bigg( \frac{r(x)^{\theta_2} + t^{\theta_1}}{R^{\theta_2}} \Bigg)^{C_1\alpha-1}
    \frac{r(x)^{\theta_2-1}}{R^{\theta_2}} \Bigg]^{\frac{p(q+\alpha)}{q-p-m+2}}  V^{-\frac{p+\alpha+m-2}{q-p-m+2}}   \;d\mu_a dt \\
    &\ \;\;\;\; \leq C |\alpha|^{\frac{p(q+\alpha)-(p-1)q}{q-p-m+2}}
    R^{C_1\theta_2|\alpha| \frac{p(q+\alpha)}{q-p-m+2}} \\
    &\ \;\;\;\; \;\;\;\;\times \int\int_{E_R^c}
    \Big[ \big( r(x)^{\theta_2} + t^{\theta_1} \big)^{1/\theta_2} \Big]^{\theta_2 (C_1\alpha-1)\frac{p(q+\alpha)}{q-p-m+2}}\\
    &\ \;\;\;\;\ \ \ \ \ \ \ \ \ \ \ \ \ \ \ \ \times r(x)^{(\theta_2-1)p \frac{q+\alpha}{q-p-m+2}}
    V^{-\frac{p+\alpha+m-2}{q-p-m+2}} \;d\mu_a dt. \\
\end{split}
\end{equation}
Now observe that for any constant $\bar{C}\in\mathbb{R}$, and for $R>1, \alpha=-\frac{1}{\log R}$, we have
\begin{equation}\label{R est}
    R^{|\alpha|\bar{C}} = e^{|\alpha|\bar{C}\log R} = e^{\bar{C}}\leq C.
\end{equation}
In addition, if $F\colon[0,\infty)\to[0,\infty)$ is decreasing and $\eqref{HP12}$ in \textbf{HP}$\mathbf{1}$ holds, then, for every $0<\epsilon<\epsilon_0$ and $R>R_0$,
\begin{equation}\label{thmf1}
\begin{split}
    \int\int_{E_R^c} F\Big( &\ \big[r(x)^{\theta_2} + t^{\theta_1} \big]^{1/\theta_2} \Big)
    r(x)^{(\theta_2-1)p\left(\frac{q}{q-p-m+2} - \epsilon\right)}
    V^{-\frac{p+m-2}{q-p-m+2} + \epsilon} \;d\mu_a dt \\
    &\ \leq C \int_{R/2^{1/\theta_2}}^{\infty} F(z) z^{\bar{s}_3 + C_0\epsilon - 1}
    \log(z)^{s_4} \;dz.
\end{split}
\end{equation}
This can be shown by minor variations in the proof of Formula $(2.19)$ in \cite{GrigS}.
Applying $\eqref{R est}$, and $\eqref{thmf1}$ with $\epsilon=\frac{|\alpha|}{q-p-m+2}$ to $\eqref{I11}$, yields
\begin{equation}\label{I12}
     |\alpha|^{-\frac{(p-1)q}{q-p-m+2}} I_1
     \leq  C  |\alpha|^{\frac{p(q+\alpha)-(p-1)q}{q-p-m+2}}
     \int_{R/2^{1/\theta_2}}^{\infty} z^{\theta_2(C_1\alpha-1)\frac{p(q+\alpha)}{q-p-m+2}+\bar{s}_3+C_0\frac{|\alpha|}{q-p-m+2}-1} \log(z)^{s_4}\;dz.
\end{equation}
Now let
\begin{equation*}
    b\coloneq \theta_2(C_1\alpha-1)\frac{p(q+\alpha)}{q-p-m+2}+\bar{s}_3+C_0\frac{|\alpha|}{q-p-m+2}.
\end{equation*}
Through the choice $C_1>\frac{C_0+\theta_2+1}{\theta_2}$, we have for $|\alpha|$ sufficiently small, i.e., for $R>1$ sufficiently large,
\begin{equation}\label{I1b}
     b < -\frac{|\alpha|}{q-p-m+2}.
\end{equation}
Integrating by substitution with $y=|b|\log z$, we can estimate the integral in $\eqref{I12}$ as follows:
\begin{equation*}
\begin{split}
    \int_{R/2^{1/\theta_2}}^{\infty} z^{\theta_2(C_1\alpha-1)\frac{p(q+\alpha)}{q-p-m+2}\bar{s}_3+C_0\frac{|\alpha|}{q-p-m+2}-1} \log(z)^{s_4}\;dz
   &\ = \int_{R/2^{1/\theta_2}}^{\infty} z^{b-1} \log(z)^{s_4}\;dz\\
    &\ \leq \int_0^{\infty} e^{-y} \Big( \frac{y}{|b|} \Big)^{s_4} \frac{1}{|b|}\;dy\\
    &\ \leq C |b|^{-s_4-1}\leq C|\alpha|^{-s_4-1},
\end{split}
\end{equation*}
where we used $\eqref{I1b}$ in the last step. In summary, we have
\begin{equation}\label{I1fin}
\begin{split}
      |\alpha|^{-\frac{(p-1)q}{q-p-m+2}} I_1 &\ \leq C   |\alpha|^{\frac{p(q+\alpha)-(p-1)q}{q-p-m+2}-s_4-1}\\
      &\ = C |\alpha|^{\frac{p(1+\alpha)+m-2}{q-p-m+2}-s_4} \leq C  |\alpha|^{\frac{p+m-2}{q-p-m+2}-s_4}.
\end{split}
\end{equation}

Let us now turn to $I_2$ in $\eqref{I_2}$. Inserting $\nabla\eta_n$ and using $ |\nabla r(x)|\leq 1$, we estimate
\begin{equation*}
    \begin{split}
        I_2 &\  \leq C \Bigg( \sup_{(x,t)\in E_{2^{1/\theta_2}nR}\setminus E_{nR}} \phi(x,t)\Bigg)^{\frac{p(q+\alpha)}{q-p-m+2}} \\
        &\;\;\;\;\times\int\int_{E_{2^{1/\theta_2}nR}\setminus E_{nR}} \Bigg[  \frac{\theta_2}{(nR)^{\theta_2}} r(x)^{\theta_2-1}\Bigg]^{\frac{p(q+\alpha)}{q-p-m+2}}   V^{-\frac{p+\alpha+m-2}{q-p-m+2}}\;d\mu_adt\\
        &\ \leq C n^{C_1\theta_2\alpha\frac{p(q+\alpha)}{q-p-m+2}} (nR)^{-\frac{\theta_2p(q+\alpha)}{q-p-m+2}}
        \int\int_{E_{2^{1/\theta_2}nR}\setminus E_{nR}} r(x)^{(\theta_2-1)p\frac{q+\alpha}{q-p-m+2}} V^{-\frac{p+\alpha+m-2}{q-p-m+2}}\;d\mu_adt.\\
    \end{split}
\end{equation*}
We can apply $\eqref{HP12}$ in \textbf{HP}$\mathbf{1}$ for $\epsilon\coloneq \frac{|\alpha|}{q-p-m+2}$ and $R$ large enough; recall that $\alpha=-\frac{1}{\log R}$. This yields
\begin{equation*}
\begin{split}
    I_2 &\ \leq C  n^{C_1\theta_2\alpha\frac{p(q+\alpha)}{q-p-m+2}} (nR)^{-\frac{\theta_2p(q+\alpha)}{q-p-m+2}}(nR)^{\bar{s}_3+C_0\frac{|\alpha|}{q-p-m+2}} \big[ \log(nR)\big]^{s_4}\\
    &\ = C n^{\frac{|\alpha|}{q-p-m+2}(-C_1\theta_2p(q+\alpha)+\theta_2p+C_0)} R^{|\alpha|\frac{\theta_2p+C_0}{q-p-m+2}}\big[ \log(nR)\big]^{s_4}.
\end{split}
\end{equation*}
Using again the lower bound on $C_1$, namely $C_1>\frac{C_0+\theta_2+1}{\theta_2}$ and the observation in $\eqref{R est}$, we have for $|\alpha|$ small enough:
\begin{equation}\label{I2fin}
    I_2 \leq C n^{-\frac{|\alpha|}{q-p-m+2}} \big[ \log(nR)\big]^{s_4}.
\end{equation}
Finally, we can return to $\eqref{thm1maineq}$ and see that there exists a constant $C$ independent on $n$ and $R$ such that
\begin{equation*}
\begin{split}
     \int\int_{E_R}  u^{q+\alpha} &\ \chi_{\beta} V \;d\mu_a dt\\
     &\ \leq \int_0^{\infty} \int_M   u^{q+\alpha} \phi_n^s\chi_{\beta} V \;d\mu_a dt  \\
    &\ \leq C \Big(  |\alpha|^{\frac{p+m-2}{q-p-m+2}-s_4} +  |\alpha|^{-\frac{(p-1)q}{q-p-m+2}} n^{-\frac{|\alpha|}{q-p-m+2}} \big[ \log(nR)\big]^{s_4}\\
    &\ \;\;\;\;\;\;\;\;\;\;\;+ |\alpha|^{\frac{1}{q-1}-s_2} +n^{-\frac{|\alpha|}{q-1}} \big[ \log(nR)\big]^{s_2} \Big).
\end{split}
\end{equation*}

\noindent Passing to the $\liminf$ as $n\rightarrow\infty$, we arrive the following inequality:
\begin{equation}\label{thm1final}
     \int\int_{E_R} u^{q+\alpha} \chi_{\beta} V \;d\mu_a dt
     \leq C \Big(  |\alpha|^{\frac{p+m-2}{q-p-m+2}-s_4} +|\alpha|^{\frac{1}{q-1}-s_2} \Big).
\end{equation}
Note that the powers of $|\alpha|$ in $\eqref{thm1final}$ are positive. So taking the limit as $R\rightarrow\infty$, i.e., $|\alpha|\rightarrow0$, by another application of Fatou's Lemma, we have
\begin{equation*}
     \int\int_{E_R} u^q \chi_{\beta} V \;d\mu_a dt
     \leq 0,
\end{equation*}
which is precisely the inequality in $\eqref{thm11}$. This finishes the proof.
\end{proof}
\subsection{Proof of Theorem~\ref{thm:Thm2}}

We start with the following intermediate result.

\begin{lem}\label{thm:Lem2}
Let $p>1$, $m\geq1$ and $q>\max(p+m-2, 1)$. Let further \[s\geq \mathrm{max}\left(1,p, \frac{2pq}{q-p-m+2},\frac{q}{q-1}\right)\] be fixed. Let $u$ be a nonnegative weak solution to Problem $\eqref{main eq}$ and let $\chi_{\beta}$ be defined as in $\eqref{chib}$. Then there exists a constant $C>0$ (depending only on $p,q,m,K$  and $s$) such that for all $\alpha \in \big(
-\frac{1}{2} \mathrm{min}\big\{1,p-1, \frac{1}{p-1},p+m-2,\frac{q-p-m+2}{p-1}\big\},0 \big)$ and all $\phi\in \mathrm{Lip}(M\times(0,\infty)), 0\leq\phi\leq1$ with compact support, one has, with $ H \coloneq\{ (x,t)\in S=M\times (0,\infty) \;:\;\phi(x.t)=1 \}$,
\begin{equation}\label{lem2result}
\begin{split}
    \int_0^{\infty}&\ \int_M u^q\phi^s\chi_{\beta}V\;d\mu_adt\\
    &\ \leq C \Bigg\{ |\alpha|^{-1-\frac{(p-1)q}{q-p-m+2}} \int\int_{S\setminus  H } |\nabla\phi|^{\frac{p(q+\alpha)}{q-p-m+2}} V^{-\frac{p+\alpha+m-2}{q-p-m+2}}  \;d\mu_a dt\\
     &\ \;\;\;\;\; \;\;\;\;\; \;  + |\alpha|^{-1} \int\int_{S\setminus H } |\partial_t\phi|^{\frac{q+\alpha}{q-1}} V^{-\frac{\alpha+1}{q-1}}  \;d\mu_a dt\Bigg\}^{\frac{p-1}{p}}\\
     &\ \;\;\;\; \times \Bigg( \int\int_{S\setminus  H } |\nabla \phi|^{\frac{pq}{q-[(1-\alpha)(p-1)+(m-1)]}}\\
     &\ \quad  \;\;\;\;\; \;\;\;\;\; \;\;\;\;\;\;\;\;\;\;\times V^{-\frac{(1-\alpha)(p-1)+(m-1)}{q-[(1-\alpha)(p-1)+(m-1)]}} \;d\mu_a dt  \Bigg)^{\frac{q-[(1-\alpha)(p-1)+(m-1)]}{pq}} \\
    &\ \;\;\;\; \times \Bigg(\int\int_{S\setminus  H }
    u^q\phi^s\chi_{\beta}V\;d\mu_a dt \Bigg)^{\frac{(1-\alpha)(p-1)+(m-1)}{pq}}\\
    &\ \;\;\;\; + C \Bigg(\int\int_{S\setminus  H }  |\partial_t\phi|^{\frac{q}{q-1}} V^{-\frac{1}{q-1}} \;d\mu_adt \Bigg)^{\frac{q-1}{q}}\\
    &\ \;\;\;\; \times  \Bigg(\int\int_{S\setminus  H } u^q\phi^s\chi_{\beta}V\;d\mu_adt \Bigg)^{\frac{1}{q}}.\\
\end{split}
\end{equation}
\begin{proof}
Since $u$ is a weak solution and $\psi\coloneq\phi^s\chi_{\beta}$ is a test function in the sense of Definition ~\ref{thm:Defn}, by $\eqref{equivdefn}$, we have
\begin{equation*}
    \begin{split}
        \int_0^{\infty}\int_M &\ u^q\phi^s\chi_{\beta}V\;d\mu_adt\\
        &\ \leq s \int_0^{\infty}\int_M  |\nabla\phi||\nabla u|^{p-1}u^{m-1}\phi^{s-1}\chi_{\beta}   f(|\nabla u|) \;d\mu_adt\\
        &\ \;\;\;\; - s \int_0^{\infty}\int_M \phi^{s-1} (\partial_t\phi) \chi_{\beta}u \;d\mu_adt\\
        &\ \;\;\;\;  - \int_0^{\infty}\int_M \phi^s \chi_{\beta}' u \;d\mu_adt\\
        &\ \revcoloneqq K_1+K_2+K_3.
    \end{split}
\end{equation*}
By the nonpositivity of $K_3$, we have in particular:
\begin{equation}\label{lem2main}
        \int_0^{\infty}\int_M u^q\phi^s\chi_{\beta}V\;d\mu_adt \leq K_1+K_2.
\end{equation}
Let us first look at $K_2$. By Hölder's Inequality, we see that
\begin{equation*}
\begin{split}
    K_2 &\ \leq s \int\int_{S\setminus  H } \Bigg( u \phi^{\frac{s}{q}}\chi_{\beta}^{\frac{1}{q}} V^{\frac{1}{q}} \Bigg) \times \Bigg( \phi^{s\frac{q-1}{q}-1} \chi_{\beta}^{\frac{q-1}{q}}  |\partial_t\phi| V^{-\frac{1}{q}} \Bigg)\;d\mu_adt\\
    &\ \leq s \Bigg(\int\int_{S\setminus  H } u^q\phi^s\chi_{\beta}V\;d\mu_adt \Bigg)^{\frac{1}{q}}\\
    &\ \;\;\;\;\; \times \Bigg(\int\int_{S\setminus  H } \phi^{s-\frac{q}{q-1}}\chi_{\beta}  |\partial_t\phi|^{\frac{q}{q-1}} V^{-\frac{1}{q-1}} \;d\mu_adt \Bigg)^{\frac{q-1}{q}}.\\
\end{split}
\end{equation*}

\noindent By the choice $s>\frac{q}{q-1}$ and since $0\leq\phi,\chi_{\beta}\leq1$, we deduce
\begin{equation}\label{lem2I2final}
\begin{split}
    K_2 \leq s \Bigg(\int\int_{S\setminus  H } u^q\phi^s\chi_{\beta}V\;d\mu_adt \Bigg)^{\frac{1}{q}}
      \times \Bigg(\int\int_{S\setminus  H }  |\partial_t\phi|^{\frac{q}{q-1}} V^{-\frac{1}{q-1}} \;d\mu_adt \Bigg)^{\frac{q-1}{q}}.
\end{split}
\end{equation}
Applying Hölder's Inequality again, this time to $K_1$, yields
\begin{equation}\label{lem2I11}
\begin{split}
    K_1 &\ = s \int_0^{\infty}\int_M \Bigg( |\nabla u|^{p-1} u^{(\alpha+m-2)\frac{p-1}{p}} \phi^{s\frac{p-1}{p}} \chi_{\beta}^{\frac{p-1}{p}} f(|\nabla u|)^{\frac{p-1}{p}} \Bigg)\\
    &\ \;\;\;\;\; \;\;\;\;\; \;\;\;\;\;  \;\;\;\;\; \times \Bigg( |\nabla\phi| u^{-(\alpha-1)\frac{p-1}{p}+(m-1)\frac{1}{p}} \phi^{\frac{s}{p}-1} \chi_{\beta}^{\frac{1}{p}}f(|\nabla u|)^{\frac{1}{p}}   \Bigg) \;d\mu_adt\\
    &\ \leq s \Bigg( \int_0^{\infty}\int_M |\nabla u|^p u^{\alpha+m-2} \phi^{s} \chi_{\beta} f(|\nabla u|) \;d\mu_adt \Bigg)^{\frac{p-1}{p}} \\
    &\ \;\;\;\;\;  \times \Bigg(\int_0^{\infty}\int_M |\nabla\phi|^p u^{-(\alpha-1)(p-1)+(m-1)} \phi^{s-p} \chi_{\beta}f(|\nabla u|)  \;d\mu_adt \Bigg)^{\frac{1}{p}}. \\
\end{split}
\end{equation}
Here we used  the convention that $\nabla u=0$ on level sets of $u$. Applying Lemma~\ref{thm:Lem1} to the first integral on the right-hand side in $\eqref{lem2I11}$, we can further estimate:
\begin{equation}\label{I1lem2}
\begin{split}
     K_1 &\ \leq C \Bigg\{ |\alpha|^{-1-\frac{(p-1)q}{q-p-m+2}} \int\int_{S\setminus  H } |\nabla\phi|^{p\frac{q+\alpha}{q-p-m+2}} V^{-\frac{p+\alpha+m-2}{q-p-m+2}}  \;d\mu_a dt\\
     &\ \;\;\;\;\; \;\;\;\;\; \; + |\alpha|^{-1} \int\int_{S\setminus  H } |\partial_t\phi|^{\frac{q+\alpha}{q-1}} V^{-\frac{\alpha+1}{q-1}}  \;d\mu_a dt\Bigg\}^{\frac{p-1}{p}}\\
     &\ \;\;\;\;\; \times \Bigg(\int\int_{S\setminus  H } |\nabla\phi|^p u^{-(\alpha-1)(p-1)+(m-1)} \phi^{s-p} \chi_{\beta}  \;d\mu_adt \Bigg)^{\frac{1}{p}}\\
     &\ \revcoloneqq K_1' \times K_1'',
\end{split}
\end{equation}
where $C$ is a suitable constant depending also on $K$, where as in our running assumption $0\le f\le K$.
In order to estimate $K_1''$ by Hölder's Inequality, we introduce
\begin{equation}\label{lem2b}
    b\coloneq \frac{q}{(1-\alpha)(p-1)+(m-1)}\;\;\;\;\;\mathrm{and} \;\;\;\;\; b'=\frac{b}{b-1}=\frac{q}{q-[(1-\alpha)(p-1)+(m-1)]}.
\end{equation}

\noindent Note that $b$ is a well-defined Hölder exponent since $q>p+m-2$ and
\begin{equation*}
    |\alpha| \leq \frac{q-p-m+2}{2(p-1)}.
\end{equation*}
Then we can estimate $K_1''$ as follows:
\begin{equation}\label{I1''}
\begin{split}
    K_1'' & = \Bigg\{\int\int_{S\setminus H } \Bigg(
    u^{(1-\alpha)(p-1)+(m-1)}\phi^{\frac{s}{b}}\chi_{\beta}^{\frac{1}{b}} V^{\frac{1}{b}} \Bigg) \\
    & \;\;\;\;\;\;\times \Bigg( |\nabla \phi|^p \phi^{\frac{s}{b'}-p} \chi_{\beta}^{\frac{1}{b'}} V^{-\frac{1}{b}} \Bigg) \;d\mu_a dt \Bigg\}^{\frac{1}{p}} \\
    & \leq \Bigg(\int\int_{S\setminus  H } u^q\phi^s\chi_{\beta}V\;d\mu_a dt \Bigg)^{\frac{1}{bp}} \\
    & \;\;\;\;\times \Bigg(\int\int_{S\setminus  H } |\nabla \phi|^{b'p} \phi^{s-b'p } \chi_{\beta} V^{-\frac{b'}{b}} \;d\mu_a dt\Bigg)^{\frac{1}{b'p}}.
\end{split}
\end{equation}
Inserting the values for $b$ and $b'$ in $\eqref{lem2b}$ into $\eqref{I1''}$, the resulting estimate into $\eqref{I1lem2}$, and using $s>\frac{2pq}{q-p-m+2}>b'p$, $0\leq\phi,\chi_{\beta}\leq1$ yields
\begin{equation}\label{lem2I1final}
\begin{split}
     K_1 \leq &\ C \Bigg\{ |\alpha|^{-1-\frac{(p-1)q}{q-p-m+2}} \int\int_{S\setminus  H } |\nabla\phi|^{\frac{p(q+\alpha)}{q-p-m+2}} V^{-\frac{p+\alpha+m-2}{q-p-m+2}}  \;d\mu_a dt \\
     &\;\;\;\;\;\;\;\;\; + |\alpha|^{-1} \int\int_{S\setminus  H } |\partial_t\phi|^{\frac{q+\alpha}{q-1}} V^{-\frac{\alpha+1}{q-1}}  \;d\mu_a dt \Bigg\}^{\frac{p-1}{p}} \\
    &\times \Bigg(\int\int_{S\setminus  H }
    u^q\phi^s\chi_{\beta}V\;d\mu_a dt \Bigg)^{\frac{(1-\alpha)(p-1)+(m-1)}{pq}} \\
    &\times \Bigg( \int\int_{S\setminus  H }  |\nabla \phi|^{\frac{pq}{q-[(1-\alpha)(p-1)+(m-1)]}}\\
    & \quad \;\;\;\;\;\;\;\;\;\;\;\;\;\;\; \times V^{-\frac{(1-\alpha)(p-1)+(m-1)}{q-[(1-\alpha)(p-1)+(m-1)]}} \;d\mu_a dt \Bigg)^{\frac{q-[(1-\alpha)(p-1)+(m-1)]}{pq}}.
\end{split}
\end{equation}
Lastly, we insert $\eqref{lem2I2final}$ and  $\eqref{lem2I1final}$ into $\eqref{lem2main}$. It follows that $\eqref{lem2result}$ does indeed hold.
\end{proof}
\end{lem}

We are now ready to prove \textbf{Theorem~\ref{thm:Thm2}}:

\begin{proof}

As in the proof of Theorem~\ref{thm:Thm1} above, we will show that for any fixed $\beta > 0$,
\begin{equation}\label{thm21aim}
     \int_0^{\infty} \int_M   u^q \chi_{\beta} V \;d\mu_a dt  = 0.
\end{equation}

\noindent Applying Fatou's Lemma as above, it follows that $u=0$ a.e. on $M\times(0,\infty)$, as claimed in the statement of this theorem.

Firstly, observe that, similarly to the observation in the proof of Theorem~\ref{thm:Thm1}, Hypothesis \textbf{HP}$\mathbf{2}$ provides, in particular, growth estimates for sets of the form $E_{2^{1/\theta_2}nR}\setminus E_{nR}$, with $\theta_2\geq1$, which will be the relevant sets appearing in the proof below.

In the following, let $C_0, \theta_1, \theta_2$ be as in \textbf{HP}$\mathbf{2}$. Let $\alpha=-\frac{1}{\log R}$. We will use $\eqref{lem2result}$ in Lemma~\ref{thm:Lem2} with the test functions $\phi_n=\phi\eta_n$, where $\eta_n$ is as in $\eqref{eta}$, and $\phi$ is as in $\eqref{phi}$ with
\begin{equation}\label{C_12}
    C_1 > \max\Bigg(\frac{C_0+\theta_2+1}{\theta_2},\frac{2C_0(p-1)q+1}{\theta_2pq(q-p-m+2)}\Bigg).
\end{equation}
Similar to the notation in Lemma~\ref{thm:Lem2}, we let $H _n\coloneq\{ (x,t)\in S=M\times (0,\infty) \;:\;\phi_n(x.t)=1 \}$. Inserting $\phi_n$ into $\eqref{lem2result}$ then yields for $|\alpha|$ small enough
\begin{equation}\label{Thm2maineq}
\begin{split}
    \int_0^{\infty}\int_M&\  u^q\phi_n^s\chi_{\beta}V\;d\mu_adt\\
         &\leq  C \Big\{ |\alpha|^{-1-\frac{(p-1)q}{q-p-m+2}} J_1 + |\alpha|^{-1} J_2\Big\}^{\frac{p-1}{p}} \times  J_3^{\frac{q-[(1-\alpha)(p-1)+(m-1)]}{pq}} \\
    &\ \times \Bigg(\int\int_{S\setminus  H _n}
    u^q\phi_n^s\chi_{\beta}V\;d\mu_a dt \Bigg)^{\frac{(1-\alpha)(p-1)+(m-1)}{pq}}\\
    &\ +C J_4^{\frac{q-1}{q}} \times \Bigg(\int\int_{S\setminus  H _n}
    u^q\phi_n^s\chi_{\beta}V\;d\mu_a dt \Bigg)^{\frac{1}{q}}\\
     \leq &\ C \Bigg[ |\alpha|^{-\frac{p-1}{p}} \Big\{ |\alpha|^{-\frac{(p-1)q}{q-p-m+2}}J_1 + J_2 \Big\}^{\frac{p-1}{p}} \times  J_3^{\frac{q-[(1-\alpha)(p-1)+(m-1)]}{pq}} +J_4^{\frac{q-1}{q}}  \Bigg]\\
    &\ \times \Bigg[ \Bigg(\int\int_{S\setminus  H _n}
    u^q\phi_n^s\chi_{\beta}V\;d\mu_a dt \Bigg)^{\frac{(1-\alpha)(p-1)+(m-1)}{pq}}\\
    &\ \;\;\;\;\;\;\;\; + \Bigg(\int\int_{S\setminus  H _n}
    u^q\phi_n^s\chi_{\beta}V\;d\mu_a dt \Bigg)^{\frac{1}{q}} \Bigg], \\
\end{split}
\end{equation}
where
\begin{align}
    J_1 & \coloneq\int\int_{S\setminus H_n} |\nabla\phi_n|^{p\frac{q+\alpha}{q-p-m+2}} V^{-\frac{p+\alpha+m-2}{q-p-m+2}}  \;d\mu_a dt, \label{J1}\\
    J_2 & \coloneq\int\int_{S\setminus H_n} |\partial_t\phi_n|^{\frac{q+\alpha}{q-1}} V^{-\frac{\alpha+1}{q-1}}  \;d\mu_a dt,  \label{J2}\\
    J_3 & \coloneq\int\int_{S\setminus H_n}  |\nabla \phi_n|^{\frac{pq}{q-[(1-\alpha)(p-1)+(m-1)]}} V^{-\frac{(1-\alpha)(p-1)+(m-1)}{q-[(1-\alpha)(p-1)+(m-1)]}} \;d\mu_a dt,  \label{J3}\\
    J_4 & \coloneq\int\int_{S\setminus H_n}  |\partial_t\phi_n|^{\frac{q}{q-1}} V^{-\frac{1}{q-1}} \;d\mu_adt.  \label{J4}
\end{align}

We begin by estimating $J_1$ in $\eqref{J1}$. Observe that, by $\eqref{a}$,
\begin{equation*}
    J_1\leq C( I_1+I_2),
\end{equation*}
with $I_1$ and $I_2$ as in $\eqref{I_1}$, $\eqref{I_2}$ in Theorem~\ref{thm:Thm1}. Proceeding as in the estimates for $I_1$ and $I_2$ but with $\eqref{HP22}$ in \textbf{HP}$\mathbf{2}$ instead of $\eqref{HP12}$ in \textbf{HP}$\mathbf{1}$, we arrive at
\begin{equation*}
    |\alpha|^{-\frac{(p-1)q}{q-p-m+2}} J_1
    \leq C(1+|\alpha|^{-\frac{(p-1)q}{q-p-m+2}}n^{-\frac{|\alpha|}{q-p-m+2}} \big[\log(nR)\big]^{\bar{s}_4}).
\end{equation*}
Indeed, we only have to replace $s_4$ in $\eqref{I1fin}$ and $\eqref{I2fin}$ with $\bar{s}_4$. We conclude that
\begin{equation}\label{j1l1}
    \limsup_{n\rightarrow\infty}  |\alpha|^{-\frac{(p-1)q}{q-p-m+2}} J_1  \leq C.
\end{equation}

Turning to $J_2$ in $\eqref{J2}$, we see that, by $\eqref{a}$,
\begin{equation*}
    J_2\leq C (I_3+I_4),
\end{equation*}
with $I_3$ and $I_4$ as in $\eqref{I_3}$, $\eqref{I_4}$ in Theorem~\ref{thm:Thm1}. Similarly to the reasoning for $J_1$ above, one can show that Estimate $\eqref{HP21}$ in \textbf{HP}$\mathbf{2}$ instead of Estimate $\eqref{HP11}$ in \textbf{HP}$\mathbf{1}$ yields
\begin{equation*}
    \begin{split}
        J_2 \leq C(1+n^{-\frac{|\alpha|}{q-1}}\big[\log(nR)\big]^{\bar{s}_2}).
    \end{split}
\end{equation*}
Indeed, we only have to replace $s_2$ in $\eqref{I34fin}$ with $\bar{s}_2$. Thus,
\begin{equation}\label{j2l1}
    \limsup_{n\rightarrow\infty}  J_2  \leq C.
\end{equation}

In order to estimate $J_4$ in $\eqref{J4}$ note that this integral coincides with the integral $J_5$ on p. $956$ in \cite{dariospaper}, with the exception that we integrate against the weighted measure $\mu_a$. Following the estimate for $J_5$ in \cite{dariospaper} step by step and inserting $\eqref{HP21}$ in \textbf{HP}$2$ for $\epsilon =0$ (see Remark ~\ref{thm:rmkepsilon}), we see that:
\begin{equation}\label{j4l1}
    \limsup_{n\rightarrow\infty} J_4 \leq C.
\end{equation}
So it remains to estimate $J_3$ in $\eqref{J3}$. To this end, we introduce
\begin{equation*}%\label{Delta}
    \delta\coloneq\frac{|\alpha|(p-1)q}{(q-p-m+2)\{q-[(1-\alpha)(p-1)+(m-1)]\}}.
\end{equation*}
Then
\begin{equation}\label{deltaest}
    0<\frac{|\alpha|(p-1)q}{(q-p-m+2)^2} < \delta  <  \frac{2|\alpha|(p-1)q}{(q-p-m+2)^2}
\end{equation}
for $0<|\alpha|<\frac{q-p-m+2}{2(p-1)}$. With this, we can rewrite the powers of $V$ and $|\nabla\phi|$ in $J_3$:
\begin{equation*}
   \frac{(1-\alpha)(p-1)+(m-1)}{q-[(1-\alpha)(p-1)+(m-1)]} = \bar{s}_4+\delta
\end{equation*}
and
\begin{equation}\label{iddeltap}
     \frac{pq}{q-[(1-\alpha)(p-1)+(m-1)]} = \frac{\bar{s}_3}{\theta_2}+p\delta.
\end{equation}

Now $J_3$ can be rewritten and bounded as follows, using $\eqref{a}$:
\begin{equation}\label{j31}
\begin{split}
       J_3 &\ =\int\int_{S\setminus  H _n}  |\nabla \phi|^{\frac{pq}{q-[(1-\alpha)(p-1)+(m-1)]}} V^{-\frac{(1-\alpha)(p-1)+(m-1)}{q-[(1-\alpha)(p-1)+(m-1)]}} \;d\mu_a dt\\
       &\ =\int\int_{S\setminus  H _n}  |\nabla \phi_n|^{\frac{\bar{s}_3}{\theta_2}+p\delta} V^{-\bar{s}_4-\delta} \;d\mu_a \, dt \\
       &\ \leq C \Bigg( \int\int_{E_R^c}   |\nabla \phi|^{\frac{\bar{s}_3}{\theta_2}+p\delta} V^{-\bar{s}_4-\delta} \;d\mu_a dt\\
       & \quad \;\;\;\;\;\;\;\;+ \int\int_{E_{2^{1/\theta_2}}nR\setminus E_{nR}}  \phi_n^{\frac{\bar{s}_3}{\theta_2}+p\delta} |\nabla \eta_n|^{\frac{\bar{s}_3}{\theta_2}+p\delta} V^{-\bar{s}_4-\delta} \;d\mu_a dt\Bigg) \\
       &\ \revcoloneqq C (I_7 + I_8).\\
\end{split}
\end{equation}

Let us first look at $I_7$. By computing the gradient of $\phi$ and using $|\nabla r(x)|\leq1$, we can estimate
\begin{equation}\label{I7first}
\begin{split}
   I_7 &\ \leq C |\alpha|^{\left( \frac{\bar{s}_3}{\theta_2} + p\delta\right)} \int\int_{E_R^c}  \left[ \frac{r(x)^{\theta_2 }+ t^{\theta_1}}{R^{\theta_2}} \right]^{(C_1\alpha-1)\left( \frac{\bar{s}_3}{\theta_2} + p\delta \right)}\\
     & \quad \;\;\;\;\,\;\;\;\;\;\;\;\;\;\;\;\;\;\;\;\;\;\;\;\;\;\;\;\;\;\;\;\;\;\;\times\left(\frac{r(x)^{(\theta_2-1)}}{R^{\theta_2}}\right)^{ \frac{\bar{s}_3}{\theta_2} + p\delta } V^{-\bar{s}_4 - \delta} \; d\mu_a dt \\
    &\ \leq C |\alpha|^{\left( \frac{\bar{s}_3}{\theta_2} + p\delta\right)}
    R^{C_1|\alpha|\left( \frac{\bar{s}_3}{\theta_2} + p\delta\right)}  \\
    &\ \;\;\;\; \times \int\int_{E_R^c}  \Big[ \big(r(x)^{\theta_2 }+ t^{\theta_1} \big)^{\frac{1}{\theta_2}} \Big]^{\theta_2(C_1\alpha-1)\left( \frac{\bar{s}_3}{\theta_2} + p\delta \right)} r(x)^{(\theta_2-1)\left( \frac{\bar{s}_3}{\theta_2} + p\delta \right)} V^{-\bar{s}_4 - \delta} \; d\mu_a dt \\
\end{split}
\end{equation}
In addition, note that if $F\colon[0,\infty)\to[0,\infty)$ is decreasing and $\eqref{HP23}$ in \textbf{HP}$\mathbf{2}$ holds, then for every $0<\epsilon<\epsilon_0$ and $R>R_0$,
\begin{equation*}%\label{thmf2}
\begin{split}
    \int\int_{E_R^c} F\Big( & \Big[ r(x)^{\theta_2} + t^{\theta_1} \Big]^{1/\theta_2} \Big) r(x)^{(\theta_2-1)\left( \frac{\bar{s}_3}{\theta_2} + p\epsilon \right)} V^{-\bar{s}_4 - \epsilon} \; d\mu_a dt \\
     & = \int\int_{E_R^c} F\Big( \Big[ r(x)^{\theta_2} + t^{\theta_1} \Big]^{1/\theta_2} \Big) r(x)^{(\theta_2-1)p\left(\frac{q}{q-p-m+2} + \epsilon\right)}  V^{-\frac{p+m-2}{q-p-m+2} - \epsilon} \;d\mu_a dt \\
    & \leq C \int_{R/2^{1/\theta_2}} F(z) z^{\bar{s}_3 + C_0\epsilon - 1}
    \log(z)^{\bar{s}_4} \;dz.
\end{split}
\end{equation*}
This can be shown through minor adjustments of Formula $(2.19)$ in \cite{GrigS}.

We apply this to $\eqref{I7first}$ with $\epsilon=\delta$, where $\delta<\epsilon_0$ for $|\alpha|$ small enough by $\eqref{deltaest}$; we further make use of the observation in $\eqref{R est}$. Then
\begin{equation}\label{i7}
\begin{split}
    I_7  &\ \leq C|\alpha|^{\left( \frac{\bar{s}_3}{\theta_2} + p\delta\right)}  \int_{R/2^{1/\theta_2}}^{\infty} z^{ \theta_2(C_1\alpha-1) \left( \frac{\bar{s}_3}{\theta_2} + p\delta \right) + \bar{s}_3 + C_0\delta - 1} \log(z)^{\bar{s}_4} \; dz \\
    &\ \revcoloneqq C|\alpha|^{\left( \frac{\bar{s}_3}{\theta_2} + p\delta\right)} I_7'.
\end{split}
\end{equation}
In order to estimate $I_7'$, let
\begin{equation*}
    a \coloneq \theta_2(C_1\alpha-1) \frac{pq}{q-[(1-\alpha)(p-1)+(m-1)]}+\bar{s}_3+C_0\delta.\\
\end{equation*}
By inserting the upper bound for $\delta$ in $\eqref{deltaest}$ and using the lower bound for $C_1$ in $\eqref{C_12}$, we can estimate
\begin{equation}\label{thm2a}
\begin{split}
    a &\ < \theta_2(C_1\alpha-1) \frac{pq}{q-p-m+2}+\frac{pq\theta_2}{q-p-m+2}+\frac{2C_0|\alpha|(p-1)q}{(q-p-m+2)^2}\\
    &\ = -\frac{|\alpha|}{(q-p-m+2)^2} \Big( C_1 \theta_2 pq(q-p-m+2)-2C_0(p-1)q \Big) \\
    &\ < -\frac{|\alpha|}{(q-p-m+2)^2} <0.
\end{split}
\end{equation}
Furthermore, by the change of variables $y=|a|\log z$, and $\eqref{thm2a}$, we see that
\begin{equation}\label{lasti7'}
\begin{split}
    I_7' &\ = \int_{R/2^{1/\theta_2}}^{\infty} z^{a - 1} \log(z)^{\bar{s}_4} \; dz  \leq \int_0^{\infty} e^{-y}\Big( \frac{y}{|a|} \Big)^{\bar{s}_4}\frac{1}{|a|}\;dy \leq C |\alpha|^{-\bar{s}_4-1}.
\end{split}
\end{equation}
In summary, we conclude from $\eqref{lasti7'}$, $\eqref{i7}$ and $\eqref{iddeltap}$
\begin{equation}\label{lastj31}
    I_7\leq C|\alpha|^{\frac{pq}{q-[(1-\alpha)(p-1)+(m-1)]}-\bar{s}_4-1}\leq C|\alpha|^{\frac{pq}{q-[(1-\alpha)(p-1)+(m-1)]}-\frac{q}{q-p-m+2}}.
\end{equation}

It remains to consider $I_8$ in $\eqref{j31}$. Recall that $|\nabla r(x)|\leq 1$ when computing $\nabla\eta_n$. Then
\begin{equation*}
\begin{split}
       I_8 &\ \leq \Big( \sup_{E_{2^{1/\theta_2}nR}\setminus E_{nR}} \phi \Big)^{\frac{\bar{s}_3}{\theta_2}+p\delta}
       \int\int_{E_{2^{1/\theta_2}nR}\setminus E_{nR}} \Bigg( \frac{\theta_2 r(x)^{\theta_2-1}}{(nR)^{\theta_2}} \Bigg)^{\frac{\bar{s}_3}{\theta_2}+p\delta}V^{-\bar{s}_4-\delta}\;d\mu_a \, dt \\
       &\ \leq C n^{\theta_2C_1\alpha\left(\frac{\bar{s}_3}{\theta_2}+p\delta\right)}
       (nR)^{-\theta_2 \left(\frac{\bar{s}_3}{\theta_2}+p\delta\right)}
        \int\int_{E_{2^{1/\theta_2}nR}\setminus E_{nR}}
        r(x)^{(\theta_2-1)\left(\frac{\bar{s}_3}{\theta_2}+p\delta\right)}V^{-\bar{s}_4-\delta}  \;d\mu_a  dt \\
        &\ = C n^{\theta_2(C_1\alpha-1)\left(\frac{\bar{s}_3}{\theta_2}+p\delta\right)}
       R^{-\theta_2\left(\frac{\bar{s}_3}{\theta_2}+p\delta\right)}\\
       &\ \;\;\;\; \times \int\int_{E_{2^{1/\theta_2}nR}\setminus E_{nR}}
        r(x)^{(\theta_2-1)p\left(\frac{q}{q-p-m+2}+\delta\right)}V^{-\frac{p+m-2}{q-p-m+2}-\delta}  \;d\mu_a dt.\\
\end{split}
\end{equation*}
By $\eqref{HP23}$ in \textbf{HP}$\mathbf{2}$ and the observation in $\eqref{R est}$, this can be estimated by
\begin{equation}\label{thm2I8}
\begin{split}
    I_8 &\ \leq C  n^{\theta_2(C_1\alpha-1)\left(\frac{\bar{s}_3}{\theta_2}+p\delta\right)}
       R^{-\theta_2\left(\frac{\bar{s}_3}{\theta_2}+p\delta\right)}(nR)^{\bar{s}_3+C_0\delta} \big[\log(nR)\big]^{\bar{s}_4}\\
       &\ = C n^{\theta_2C_1\alpha\left(\frac{\bar{s}_3}{\theta_2}+p\delta\right)-\theta_2p\delta+C_0\delta} R^{(-\theta_2p+C_0)\delta}  \big[\log(nR)\big]^{\bar{s}_4}\\
       &\ \leq C n^{\theta_2C_1\alpha\left(\frac{\bar{s}_3}{\theta_2}+p\delta\right)-\theta_2p\delta+C_0\delta} \big[\log(nR)\big]^{\bar{s}_4}.\\
\end{split}
\end{equation}
The power of $n$ in $\eqref{thm2I8}$ can be estimated using the bounds for $\delta$ in $\eqref{deltaest}$ and the assumption on $C_1$ in $\eqref{C_12}$; recall the identity in $\eqref{iddeltap}$:
\begin{equation*}
\begin{split}
    &\ \frac{C_1\theta_2p  q\alpha}{q-[(1-\alpha)(p-1)+(m-1)]}  -\theta_2p\delta + C_0\delta\\
     &\  \;\;\;\;\;\;\;\;\;\;\;\;\;\;\;\; <  \frac{C_1\theta_2p  q\alpha}{q-p-m+2}  + C_0\delta \\
    &\ \;\;\;\;\;\;\;\;\;\;\;\;\;\;\;\;  < -\frac{|\alpha|}{(q-p-m+2)^2} \Big( C_1\theta_2pq (q-p-m+2)-2C_0(p-1)q\Big)\\
    &\ \;\;\;\;\;\;\;\;\;\;\;\;\;\;\;\;  < -\frac{|\alpha|}{(q-p-m+2)^2} <0.\\
\end{split}
\end{equation*}
Thus, $\eqref{thm2I8}$ becomes
\begin{equation}\label{lastj32}
    I_8 \leq C n^{-\frac{|\alpha|}{(q-p-m+2)^2}} \big[\log(nR)\big]^{\bar{s}_4}.
\end{equation}

By combining $\eqref{lastj31}$ and $\eqref{lastj32}$, we arrive at
\begin{equation*}
    J_3\leq C\Big(|\alpha|^{\frac{pq}{q-[(1-\alpha)(p-1)+(m-1)]}-\frac{q}{q-p-m+2}}  + n^{-\frac{|\alpha|}{(q-p-m+2)^2}} \big[\log(nR)\big]^{\bar{s}_4}\Big),
\end{equation*}
and hence,
\begin{equation}\label{j3further}
\begin{split}
    \limsup_{n\rightarrow\infty} J_3 &\ \leq C|\alpha|^{\frac{pq}{q-[(1-\alpha)(p-1)+(m-1)]}-\frac{q}{q-p-m+2}}\\
    &\ = C |\alpha|^{\frac{(p-1)q}{q-[(1-\alpha)(p-1)+(m-1)]}+\frac{|\alpha|q(p-1)}{\{q-[(1-\alpha)(p-1)+(m-1)]\}(q-p-m+2)}}\\
    &\ \leq C |\alpha|^{\frac{(p-1)q}{q-[(1-\alpha)(p-1)+(m-1)]}}.
\end{split}
\end{equation}

Now, return to Inequality $\eqref{Thm2maineq}$ and notice:
Since $p>1$ and $q>\max(p+m-2, 1)$, there exists a $\gamma\in(0,1)$ such that for every sufficiently small $|\alpha|$, we have
\begin{equation*}
    0<\frac{(1-\alpha)(p-1)+(m-1)}{pq}<\gamma \;\;\;\mathrm{and}\;\;\; 0<\frac{1}{q}<\gamma.
\end{equation*}

Therefore, from $\eqref{Thm2maineq}$, we obtain
\begin{equation}\label{1701}
\begin{split}
    \int_0^{\infty}\int_M &\  u^q\phi_n^s\chi_{\beta}V\;d\mu_adt \\
     \leq & C \Bigg[ |\alpha|^{-\frac{p-1}{p}} \Big\{ |\alpha|^{-\frac{(p-1)q}{q-p-m+2}}J_1 + J_2 \Big\}^{\frac{p-1}{p}}\, J_3^{\frac{q-[(1-\alpha)(p-1)+(m-1)]}{pq}} +J_4^{\frac{q-1}{q}}  \Bigg]\\
    &\ \times\Bigg(1 + \int_0^{\infty}\int_M
    u^q\phi_n^s\chi_{\beta}V\;d\mu_a dt \Bigg)^{\gamma}. \\
\end{split}
\end{equation}

Observe that, by the Monotone Convergence Theorem,
\begin{equation}\label{introj}
    J \coloneq \lim_{n\rightarrow\infty} \int_0^{\infty}\int_M  u^q\phi_n^s\chi_{\beta} V\;d\mu_adt = \int_0^{\infty}\int_M  u^q\phi^s\chi_{\beta} V\;d\mu_adt \in [0,\infty].
\end{equation}

In addition, by $\eqref{j1l1}$, $\eqref{j2l1}$, $\eqref{j4l1}$ and $\eqref{j3further}$
\begin{equation}\label{sup42}
    \limsup_{n\rightarrow\infty}\Bigg[|\alpha|^{-\frac{p-1}{p}} \Big\{ |\alpha|^{-\frac{(p-1)q}{q-p-m+2}}J_1 + J_2 \Big\}^{\frac{p-1}{p}} \times  J_3^{\frac{q-[(1-\alpha)(p-1)+(m-1)]}{pq}} +J_4^{\frac{q-1}{q}}\Bigg]\leq C.
\end{equation}

Combining $\eqref{1701}$, $\eqref{introj}$ and $\eqref{sup42}$, we arrive at
\begin{equation*}
    J\leq C(1+J)^{\gamma}.
\end{equation*}
In particular, as $\gamma\in(0,1)$,
\begin{equation}\label{jfin}
    J= \int_0^{\infty}\int_M  u^q\phi^s\chi_{\beta} V\;d\mu_adt\leq C
\end{equation}
for some $C$ independent of $\alpha,\beta$ and $R$ (for sufficiently large $R>1$, i.e., $|\alpha|$ small enough).

We now proceed to prove that
\begin{equation*}
    \int_0^{\infty}\int_M  u^q\chi_{\beta} V\;d\mu_adt = 0.
\end{equation*}

Recall that $H _n=\{ (x,t)\in S=M\times (0,\infty) \;:\;\phi_n(x.t)=1 \}\supset E_R$. From $\eqref{Thm2maineq}$ we obtain
\begin{equation}\label{1057}
\begin{split}
    \int\int_{E_R} &\  u^q\chi_{\beta}V\;d\mu_adt\\
     \leq & \int_0^{\infty}\int_M  u^q\phi_n^s\chi_{\beta} V\;d\mu_adt\\
     \leq &\ C \Bigg[ |\alpha|^{-\frac{p-1}{p}} \Big\{ |\alpha|^{-\frac{(p-1)q}{q-p-m+2}}J_1 + J_2 \Big\}^{\frac{p-1}{p}} \times  J_3^{\frac{q-[(1-\alpha)(p-1)+(m-1)]}{pq}} +J_4^{\frac{q-1}{q}}  \Bigg]\\
    &\ \times \Bigg[ \Bigg(\int\int_{E_R^c}
    u^q\phi_n^s\chi_{\beta}V\;d\mu_a dt \Bigg)^{\frac{(1-\alpha)(p-1)+(m-1)}{pq}}\\
    &\ \;\;\;\;\;\;\;\; + \Bigg(\int\int_{E_R^c}
    u^q\phi_n^s\chi_{\beta}V\;d\mu_a dt \Bigg)^{\frac{1}{q}} \Bigg].\\
\end{split}
\end{equation}

Note that, by the Monotone Convergence Theorem,
\begin{equation*}
    \lim_{n\rightarrow\infty} \int\int_{E_R^c}
    u^q\phi_n^s\chi_{\beta}V\;d\mu_a dt = \int\int_{E_R^c}
    u^q\phi^s\chi_{\beta}V\;d\mu_a dt.
\end{equation*}
Combining this with $\eqref{sup42}$ in $\eqref{1057}$, we have for sufficiently large $R>1$ and any $\beta>0$, with $C$ independent of $R,\alpha$ and $\beta$,
\begin{equation*}\begin{aligned}
    &\int\int_{E_R}
    u^q\chi_{\beta}V\;d\mu_a dt\\
    & \leq C \Bigg[ \Bigg(\int\int_{E_R^c}
    u^q\phi^s\chi_{\beta}V\;d\mu_a dt \Bigg)^{\frac{(1-\alpha)(p-1)+(m-1)}{pq}} + \Bigg(\int\int_{E_R^c}
    u^q\phi^s\chi_{\beta}V\;d\mu_a dt \Bigg)^{\frac{1}{q}} \Bigg].
\end{aligned}\end{equation*}
Finally, we take the limit as $R\rightarrow\infty$. Using $\eqref{jfin}$, we conclude that
\begin{equation*}
     \int_0^{\infty}\int_M  u^q\chi_{\beta}V\;d\mu_adt = 0.
\end{equation*}
This is precisely Identity $\eqref{thm21aim}$. Therefore, $u=0$ a.e. on $M\times (\beta,\infty)$ for all $\beta>0$. It follows that $u=0$ a.e. on $M\times(0,\infty)$. This completes the proof of Theorem ~\ref{thm:Thm2}.
\end{proof}

\section{Proof of Corollaries ~\ref{thm:cor1}--~\ref{thm:lastcor'}}

\begin{proof}[Proof of Corollary~\ref{thm:cor1}]
We estimate
    \begin{equation*}
    \begin{split}
        \int\int_{E_{2^{1/\theta_2}R\setminus E_R}} t^{(\theta_1-1)\left( \frac{q}{q-1}-\epsilon \right)}\;d\mu_a dt
        &\ \leq \Big(\int_0^{2^{1/\theta_1}R^{\theta_2/\theta_1}} t^{(\theta_1-1)\left( \frac{q}{q-1}-\epsilon \right)} \;dt \Big) \Big( \int_{B_{2^{1/\theta_2}R}} 1 \;d\mu_a \Big) \\
        &\ \leq C R^{\frac{\theta_2}{\theta_1}\left( (\theta_1-1)\left( \frac{q}{q-1}-\epsilon \right) +1 \right)} R^{N} .
        \end{split}
    \end{equation*}
Then $\eqref{HP11}$ in \textbf{HP}$\mathbf{1}$ is satisfied for any choice of $C_0,\epsilon_0>0, R_0>1$ and for
\begin{equation*}%\label{coratheta1}
    \frac{\theta_2}{\theta_1}\geq (q-1)N.
\end{equation*}
Similarly,
\begin{equation*}
\begin{split}
    &\int\int_{E_{2^{1/\theta_2}R\setminus E_R}} |x|^{(\theta_2-1)p\left( \frac{q}{q-p-m+2}-\epsilon \right)} \;d\mu_a dt\\
    &\ \leq \Big(\int_0^{2^{1/\theta_1}R^{\theta_2/\theta_1}} 1 \;dt \Big) \Big( \int_{B_{2^{1/\theta_2}R}} |x|^{(\theta_2-1)p\left( \frac{q}{q-p-m+2}-\epsilon \right)} \;d\mu_a \Big) \\
    &\ \leq C R^{\frac{\theta_2}{\theta_1}} \int_0^{2^{1/\theta_2}R} r^{(\theta_2-1)p\left( \frac{q}{q-p-m+2}-\epsilon \right)+N-1}\;dr\\
    &\ = C R^{\frac{\theta_2}{\theta_1}} R^{(\theta_2-1)p\left( \frac{q}{q-p-m+2}-\epsilon \right)+N}.
\end{split}
\end{equation*}
Then $\eqref{HP12}$ in \textbf{HP}$\mathbf{1}$ is satisfied for any choice of $C_0,\epsilon_0>0, R_0>1$ and
\begin{equation*}%\label{cor1theta2}
    \frac{\theta_2}{\theta_1}\leq \frac{pq}{q-p-m+2}-N.
\end{equation*}
Hence, Theorem~\ref{thm:Thm1} is applicable and nonexistence follows if $q>\max(1,p+m-2)$ and
\begin{equation*}
    q\leq \frac{p}{N}+p+m-2.
\end{equation*}
\end{proof}

\begin{proof}[Proof of Corollary~\ref{thm:cor2}]

Applying $\eqref{cor2vsep}$, $\eqref{cor2fhest}$ and $\eqref{cor2hint}$ in this order, we see that, for $R$ large enough and $\frac{1}{q-1}>\epsilon>0$
\begin{equation*}%\label{cor2comp1}
\begin{split}
    \int\int_{E_{2^{1/\theta_2}R \setminus E_R}} &\
    t^{(\theta_1-1)\left( \frac{q}{q-1}-\epsilon \right)} V^{-\frac{1}{q-1}+\epsilon} \; d\mu_a \, dt \\
    &\leq  \Bigg( \int_0^{2^{1/\theta_1}R^{\theta_2/\theta_1}} t^{(\theta_1-1)\left( \frac{q}{q-1}-\epsilon \right)}
    f(t)^{-\frac{1}{q-1}+\epsilon} \; dt \Bigg) \Bigg( \int_{B_{2^{1/\theta_2}R}} h(x)^{-\frac{1}{q-1}+\epsilon} \; d\mu_a \Bigg) \\
    &\leq C R^{\frac{\theta_2}{\theta_1}(\theta_1-1)\left( \frac{q}{q-1}-\epsilon \right)+\frac{\theta_2}{\theta_1}\alpha_2\epsilon+\alpha_1\epsilon} \\
    &\quad \times \Bigg( \int_0^{2^{1/\theta_1}R^{\theta_2/\theta_1}} f(t)^{-\frac{1}{q-1}} \; dt \Bigg) \Bigg( \int_{B_{2^{1/\theta_2}R}} h(x)^{-\frac{1}{q-1}} \; d\mu_a \Bigg) \\
    &\leq C R^{\frac{\theta_2}{\theta_1}(\theta_1-1)\left( \frac{q}{q-1}-\epsilon \right)+\frac{\theta_2}{\theta_1}\alpha_2\epsilon+\alpha_1\epsilon+\frac{\theta_2}{\theta_1}\sigma_2+\sigma_1}
    (\log R)^{\delta_1+\delta_2}.
\end{split}
\end{equation*}
Then $\eqref{HP11}$ in \textbf{HP}$\mathbf{1}$ is satisfied for any
$\theta_1,\theta_2\geq1$, $C_0>\max\big(0,\frac{\theta_2}{\theta_1}(\alpha_2+1)+\alpha_1-\theta_2\big)$ and
\begin{equation}\label{cor2res2}
    \delta_1+\delta_2 < \bar{s}_2=\frac{1}{q-1}\;\;\;\;\;\mathrm{and}\;\;\;\;\;\frac{\theta_2}{\theta_1}\left( \sigma_2-\frac{q}{q-1}
 \right)+\sigma_1\leq0.
\end{equation}
Similarly, applying $\eqref{cor2vsep}$, $\eqref{cor2fhest}$ and $\eqref{cor2fint}$ in this order, we see that, for $R$ large enough and any $\frac{p+m-2}{q-p-m+2}>\epsilon>0$:
\begin{equation*}%\label{cor2comp2}
\begin{split}
     \int \int_{E_{2^{1/\theta_2} R\setminus E_R}} &\ r(x)^{(\theta_2-1)p\left( \frac{q}{q-p-m+2}-\epsilon \right)} V^{-\frac{p+m-2}{q-p-m+2}+\epsilon} \;d\mu_a dt\\
     &\ \leq \Bigg(\int_0^{2^{1/\theta_1}R^{\theta_2/\theta_1}} f(t)^{-\frac{p+m-2}{q-p-m+2}+\epsilon}\;dt \Bigg) \\
     &\ \;\;\;\; \times\Bigg( \int_{B_{2^{1/\theta_2}R}} r(x)^{(\theta_2-1)p\left( \frac{q}{q-p-m+2}-\epsilon \right)} h(x)^{-\frac{p+m-2}{q-p-m+2}+\epsilon} \;d\mu_a \Bigg)\\
     &\ \leq C R^{\frac{\theta_2}{\theta_1}\alpha_2\epsilon+(\theta_2-1)p\left( \frac{q}{q-p-m+2}-\epsilon \right)+\alpha_1\epsilon} \Bigg(\int_0^{2^{1/\theta_1}R^{\theta_2/\theta_1}} f(t)^{-\frac{p+m-2}{q-p-m+2}}\;dt \Bigg) \\
     &\   \;\;\;\;\times\Bigg( \int_{B_{2^{1/\theta_2}R}}  h(x)^{-\frac{p+m-2}{q-p-m+2}} \;d\mu_a \Bigg)\\
     &\ \leq CR^{\frac{\theta_2}{\theta_1}\alpha_2\epsilon+(\theta_2-1)p\left( \frac{q}{q-p-m+2}-\epsilon \right)+\alpha_1\epsilon+\frac{\theta_2}{\theta_1}\sigma_4+\sigma_3} (\log R)^{\delta_3+\delta_4}.
\end{split}
\end{equation*}
So $\eqref{HP12}$ in \textbf{HP}$\mathbf{1}$ is satisfied for any
$\theta_1,\theta_2\geq1$, $C_0>\max\big(0,\frac{\theta_2}{\theta_1}\alpha_2-(\theta_2-1)p+\alpha_1\big)$ and
\begin{equation}\label{cor2res1}
    \delta_3+\delta_4 < \bar{s}_4=\frac{p+m-2}{q-p-m+2}\;\;\;\;\;\mathrm{and}\;\;\;\;\;\frac{\theta_2}{\theta_1}\sigma_4+\left(\sigma_3-\frac{pq}{q-p-m+2}\right)\leq0.
\end{equation}
For conditions $\eqref{cor2res2}$ and $\eqref{cor2res1}$ to hold, by our assumptions, it is sufficient to choose $\theta_1,\theta_2\geq1$ such that
\begin{equation*}
   \sigma_1 \left(\frac{q}{q - 1} - \sigma_2\right)^{-1} \leq \frac{\theta_2}{\theta_1}
\quad \text{if} \quad 0 \leq \sigma_2 < \frac{q}{q - 1}
\end{equation*}
and
\begin{equation*}
    \frac{\theta_2}{\theta_1} \leq \left(\frac{pq}{q - p -m+2} - \sigma_3\right) \sigma_4^{-1}
\quad \text{if} \quad 0 \leq \sigma_3 < \frac{pq}{q - p -m+2}.
\end{equation*}
Theorem~\ref{thm:Thm1} completes the proof.
\end{proof}

\begin{proof} [Proof of Corollary~\ref{thm:cor3}]
The proof follows the same strategy as in Corollary~\ref{thm:cor2}, now applying Theorem~\ref{thm:Thm2}.
\end{proof}

\begin{proof}[Proof of Corollary~\ref{thm:cor2'}]
This is an immediate consequence of Corollary~\ref{thm:cor2}, obtained by inserting $f\equiv1$ and $\sigma_2=\sigma_4=1,\delta_1=\delta_4=0$.
\end{proof}

\begin{proof}[Proof of Corollary~\ref{thm:cor3'}]
This is an immediate consequence of Corollary~\ref{thm:cor3}, obtained by inserting $f\equiv1$ and $\sigma_2=\sigma_4=1,\delta_1=\delta_4=0$.
\end{proof}

\begin{proof}[Proof of Corollary~\ref{thm:lastcor'}]
This is an immediate consequence of Corollary~\ref{thm:cor3'}, obtained by inserting $h\equiv1$.
\end{proof}

\par\bigskip\noindent
\textbf{Acknowledgments.}
The authors are members of the Gruppo Nazionale per l'Analisi Matematica, la Probabilit\`a e le loro Applicazioni (GNAMPA, Italy) of the Istituto Nazionale di Alta Matematica (INdAM, Italy). The first author is partially supported by the PRIN project 2022 ''Partial differential equations and related geometric-functional inequalities", ref. 20229M52AS. The second author is partially supported by the PRIN projects 2022 Geometric-analytic methods for PDEs
and applications, ref. 2022SLTHCE. Both PRIN projects above are
financially supported by the EU, in the framework of the "Next Generation EU initiative".

%\addcontentsline{toc}{section}{References}
%\bibliographystyle{plain}
%\bibliography{bibliography.bib}

\end{document}